\documentclass{article}

\usepackage[margin=1in]{geometry}  
\usepackage{graphicx}              
\usepackage{amsmath}               
\usepackage{amsfonts}              
\usepackage{amsthm}                
\usepackage{mathtools}
\usepackage{verbatim}
\usepackage{xifthen}
\usepackage{amssymb}
\usepackage{xstring}
\usepackage{mathabx}
\usepackage{array}  
\usepackage{booktabs}
\usepackage[all]{xy}
\usepackage{hyperref} 
\hypersetup{
    colorlinks=true,
    linkcolor=blue,
    filecolor=magenta,      
    urlcolor=cyan,
}
\usepackage{tikz-cd}
\usepackage{stmaryrd} 
\usepackage{parskip}
\usepackage{scalerel,stackengine}

\newtheorem{thm}{Theorem}[section]
\newtheorem*{theorem*}{Theorem}
\newtheorem{lem}[thm]{Lemma}
\newtheorem{prop}[thm]{Proposition}
\theoremstyle{definition}
\newtheorem{defn}{Definition}[section]

\newtheorem*{exmp}{Example}
\newtheorem{cor}{Corollary}[section]
\newtheorem*{remark}{Remark}

\newtheorem*{question}{Question}

\newcommand{\dotifempty} [1]{\ifthenelse{\isempty{#1}}
                          	{\cdot}%
                          	{#1}}

\DeclarePairedDelimiter\absolute{\lvert}{\rvert}%
\DeclarePairedDelimiter\nrm{\lVert}{\rVert}%
\DeclarePairedDelimiter\floor{\lfloor}{\rfloor}

\DeclarePairedDelimiter\braces{\{}{\}}

\newcommand{\sog}[1]{\left( #1 \right)}

\newcommand{\norm}[1]{\nrm*{\dotifempty{#1}}}                              
\newcommand{\supnorm}[1]{\norm{#1}_{\infty}}
\newcommand{\supnrm}[1]{\nrm{#1}_{\infty}}
\newcommand{\wedgenorm}[1]{\norm{#1}^{\wedge}}
\newcommand{\abs}[1]{\absolute{\dotifempty{#1}}}

\newcommand{\pabs}[1]{\absolute*{\dotifempty{#1}}_p}
                         
\newcommand{\map}[0]{\rightarrow}

\newcommand{\xmap}[0]{\xrightarrow}

\newcommand{\gorer}[0]{\ \Longrightarrow \ }

\newcommand{\quot}[2]{\left.\raisebox{.3em}{$#1$}/\raisebox{-.3em}{$#2$}\right.}

\newcommand{\limit}[1]{\lim_{#1\map\infty}}
\renewcommand{\restriction}[2]{{
  \left.\kern-\nulldelimiterspace 
  #1 
  \vphantom{\big|} 
  \right|_{#2} 
  }}

\newcommand{\N}[0]{\fld{N}}
\newcommand{\Z}[0]{\fld{Z}}
\newcommand{\Q}[0]{\fld{Q}}
\newcommand{\R}[0]{\fld{R}}
\newcommand{\C}[0]{\fld{C}}

\newcommand{\twomat}[4]{\begin{pmatrix}#1 & #2 \\ #3 & #4 \end{pmatrix}}

\newcommand{\fld}[1]{
						\ifthenelse{\isempty{#1}}
                          	{\mathbb{Q}_p}
                          	{\mathbb{#1}}
                         } 
\renewcommand{\O}[1]{
						\ifthenelse{\isempty{#1}}
                          	{\fld{Z}_p}
                          	{\mathcal{O}_{\fld{#1}}}
                         } 

\newcommand{\oneOkMesh}[1]{
						\ifthenelse{\isempty{#1}}
                          	{\mathcal{O}_{\fld{K}}}%
                          	{T^{#1}\cdot\mathcal{O}_{\fld{K}}}
                         }

\newcommand{\oneOk}[2]{
						\ifthenelse{\isempty{#1}}
                          	{\textbf{1}_{\oneOkMesh{#2}}}%
                          	{\textbf{1}_{#1 +\oneOkMesh{#2}}}
                         }

\newcommand{\oneZpMesh}[1]{
						\ifthenelse{\isempty{#1}}
                          	{\fld{Z}_p}%
                          	{p^{#1}\cdot\fld{Z}_p}
                         }

\newcommand{\oneZp}[2]{
						\ifthenelse{\isempty{#1}}
                          	{\textbf{1}_{\oneZpMesh{#2}}}%
                          	{\textbf{1}_{#1 +\oneZpMesh{#2}}}
                         }

\stackMath
\newcommand\reallywidehat[1]{%
\savestack{\tmpbox}{\stretchto{%
  \scaleto{%
    \scalerel*[\widthof{\ensuremath{#1}}]{\kern.1pt\mathchar"0362\kern.1pt}%
    {\rule{0ex}{\textheight}}
  }{\textheight}%
}{2.4ex}}%
\stackon[-6.9pt]{#1}{\tmpbox}%
}

\newcommand{\qbinom}[3]{{#1 \brack #2}_{#3}}
\newcommand{\heis}[0]{\mathcal{H}}
\newcommand{\schw}[0]{\mathcal{S}}
 
\newcommand{\dotprod}[0]{\raisebox{-0.3ex}{\scalebox{1.5}{$\cdot$}}}
\newcommand{\eps}[0]{\epsilon}
\newcommand{\Norms}[0]{\mathcal{N}}
\newcommand{\invNorms}[0]{\mathcal{N}(\schw)^{\heis}}
\newcommand{\invHomothety}[0]{\mathcal{N}_H(\schw)^{\heis}}
\newcommand{\completion}[2]{{#1}_{#2}}
\newcommand{\Grassmannian}[0]{Gr}
\newcommand{\Ucomp}[1]{\widehat{#1}}
\newcommand{\dominated}[0]{\preceq}

\renewcommand{\footnoterule}{%
  \kern -3pt
  \hrule width 2in
  \kern 2pt
}
\makeatletter
\def\blfootnote{\xdef\@thefnmark{}\@footnotetext}
\makeatother

\graphicspath{ {images/} }

\newcommand{\pcoeff}[2]{[#1,#2]} 

\title{On the rigidity of invariant norms on the $p$-adic Schrödinger representation}
\author{Amit Ophir}

\begin{document}
\maketitle
\begin{abstract}
Motivated by questions about $\C_p$-valued Fourier transform on the locally compact group $(\Q_p^d,+)$, we study invariant norms on the $p$-adic Schrödinger representation of the Heisenberg group.
Our main result is a minimality and rigidity property for norms in a family of invariant norms parameterized by a Grassmannian.
This family is the orbit of the sup norm under the action of the symplectic group, acting via intertwining operators.
We also prove general fundamental properties of quotients of the universal unitary completion of cyclic algebraic representations.
Combined with the rigidity property, we are able to show that the completion of the Schrödinger representation in any of the norms in that family satisfies a strong notion of irreducibility and a strong version of Schur's lemma.
Norms that can be formed as the maximum of a finite number of norms from that family are also studied.
We conclude this paper with a list of open questions.
\end{abstract}

\tableofcontents


\section{Introduction}
Choose a non-trivial smooth character $\psi:(\Q_p,+)\map \C_p^\times$.
The $\C_p$-valued Haar distribution $dt$ on $\Q_p$ is not a measure.
As a result, we cannot define the integral $\intop f(t)\ dt$ for a general continuous function $f:\Q_p\map\C_p$, even if $f$ is compactly supported.
However, we can define integration for locally constant functions with compact support, and we denote by $\schw(\Q_p)$ the space of all such functions.

The Fourier transform of $f\in \schw(\Q_p)$ is defined by
\[\reallywidehat{f}(x)=\intop_{\Q_p}\psi(xt)f(t)\ dt.\]

The Fourier transform is not continuous in the sup norm.
In \cite{ophir2016q}, we showed that the Fourier transform is "as discontinuous as it can get" in the sense that the graph
\[\Gamma=\braces{(f,\reallywidehat{f})\ |\ f\in\schw(\Q_p)}\]
is dense in $C_0(\Q_p)\times C_0(\Q_p)$.
Here, $C_0(\Q_p)$ is the space of continuous functions that go to zero at infinity, the completion of $\schw(\Q_p)$ in the sup norm. 
The proof in \cite{ophir2016q} went by restricting the Fourier transform to some finite dimensional subspaces and used a special decomposition of the Fourier transform on these subspaces. 
It also used $q$-arithmetic, but in an entirely different way than the way $q$-arithmetic is used in the present paper.

To motivate the results in this paper, we describe two other approaches to the discontinuity of the Fourier transform.
By introducing the Heisenberg group and the Schrödinger representation, we can reformulate the above result in terms of invariant norms.

The Heisenberg group $\heis_3(\Q_p)$ is the group of unipotent matrices
\[\heis_3(\Q_p)=\braces*{
\begin{pmatrix}
1 & a & t\\
0 & 1 & b\\
0 & 0 & 1
\end{pmatrix}
}\subset GL_3(\Q_p).\]
We denote its elements by $[a,b,t]$.
The (smooth) Schrödinger representation $\rho_\psi:\heis_3(\Q_p)\map GL(\schw(\Q_p))$, attached to the character $\psi$ is defined by
\[\sog{\rho_\psi([a,b,t])f}(x)=\psi\sog{t+\frac{ab}{2}}\cdot \psi(bx)\cdot f(x+a).\]
The representation $\rho_\psi$ is irreducible and the Stone-von Neumann theorem says that, up to isomorphism, $\rho_\psi$ is the unique smooth irreducible representation of $\heis_3(\Q_p)$ with central character $\psi$.

The sup norm is invariant under the action of the Heisenberg group and so is the norm $\wedgenorm{f}:=\supnrm{\reallywidehat{f}}$.
We remark that these two norms are not equivalent.

Let $\Lambda$ and $\reallywidehat{\Lambda}$ be the closed unit balls of $\supnorm{}$ and $\wedgenorm{}$ respectively.
It is an easy exercise to prove that the following are equivalent.
    \begin{enumerate}
    \item The graph $\Gamma$ is dense in $C_0(\Q_p)\times C_0(\Q_p)$.
    \item $\Lambda+\reallywidehat{\Lambda}=\schw(\Q_p)$.
    \item There exists no $\heis_3(\Q_p)$-invariant norm that is smaller than both $\supnorm{}$ and $\wedgenorm{}$.
    \end{enumerate}

It turns out that $(3)$ is true because $\supnorm{}$ (and likewise $\wedgenorm{}$) is a minimal $\heis_3(\Q_p)$-invariant norm.
In addition, it has a surprising rigidity.
This is the content of Theorem \ref{thm_strong_minimality}, which in this case says the following.

\begin{theorem*}
Let $\norm{}$ be an $\heis_3(\Q_p)$-invariant norm on $\schw(\Q_p)$ that is dominated by the sup norm (i.e. $\norm{}\leq c\cdot \supnorm{}$ for some $c>0$).
Then there exists $r>0$ such that $\norm{}=r\cdot \supnorm{}$.
\end{theorem*} 
Clearly, $(3)$  follows.

Yet another way to approach the question of the density of the graph $\Gamma$ is to consider the intersection
\[W=\overline{\Gamma}\cap\sog{C_0(\Q_p)\times\braces{0}}.\]
Viewing $W$ as a subspace of $C_0(\Q_p)$, it is a closed subspace and is invariant by the action of the Heisenberg group.
In \cite{demathan}, Fresnel and de Mathan constructed a non-zero element in $W$.
Thus, showing that $C_0(\Q_p)$ is topologically irreducible as a representation of $\heis_3(\Q_p)$ gives another proof of the density of $\Gamma$.
In this paper we prove that $C_0(\Q_p)$ is topologically irreducible.
In fact, we will show (Proposition \ref{topologically_irreducible}) that a stronger notion of irreducibility  holds for $C_0(\Q_p)$ (see Definitions \ref{def_strong_irreducible} and \ref{def_srongly_cyclic_vector}).

The results of this paper are more general than the above discussion in two ways.
First, we work with the group $(\Q_p^d,+)$, where $d\geq 1$ is an integer, and correspondingly, with higher dimensional Heisenberg groups.
Second, we consider all the intertwining operators on the Schrödinger representation, among which the Fourier transform is just a single example.
This allows us to study simultaneous continuity of any finite number of intertwining operators (see section $7$).

The methods of the proofs are of two types.
There are general results on Banach representations and $p$-adic functional analysis.
These are contained in section $3$.
The other type is $q$-arithmetic.
More precisely, we use $q$-Mahler bases in $C(\Z_p)$ and $p$-adic evaluations of some $q$-analog expressions in order to study norms on $C(\Z_p)$.

By using the results of section $3$, it can be shown that the local maximality (Definition \ref{def_local_maximality}) of the sup norm on $C(\Z_p)$ with respect to multiplication by smooth characters is equivalent to Theorem $2$ in \cite{demathan}.
In Section \ref{subsection_new_proof_Fresnel_de_Mathan} we use our methods to give a new proof of the main results in \cite{demathan}.
Our proof, using $q$-arithmetic, can be generalized to include the case where $\psi:(\Q_p,+)\map \C_p^\times$ is continuous but not smooth, and this case does not follow from \cite{demathan}.
These results will appear in a forthcoming paper.

We remark that the completions that we study of the Schrödinger representation are large in the sense that the reduction of their unit ball modulo the maximal ideal of $\mathcal{O}_{\C_p}$ is a non-admissible smooth representation over $\overline{\fld{F}}_p$.
\subsection*{Acknowledgement}
The author is grateful to Ehud de Shalit for many helpful discussions and for reading and commenting on the first draft of this paper.

    \subsection*{List of notation}
$p$ is a prime number and we fix an algebraic closure $\Q_p^{al}$ of $\Q_p$.
The absolute value $\abs{}_p$ on $\Q_p$ extends uniquely to $\Q_p^{al}$ and we denote by $\C_p$ the completion of $\Q_p^{al}$ with respect to $\abs{}_p$.
The field $\C_p$ is a complete non-archimedean normed field and algebraically closed.
We denote by $\mathcal{O}_{\C_p}$ the set of elements $a\in\C_p$ with $\pabs{a}\leq 1$.
    \begin{itemize}
    \item[]$d$ - A fixed integer, $d\geq 1$.
    \item[]$\heis=\heis_{2d+1}(\Q_p)$ - The $2d+1$-dimensional Heisenberg group over $\Q_p$.
    \item[]$Sp_{2d}(\Q_p)$ - The $2d$-dimensional symplectic group.
    \item[]$\Grassmannian$ - The quotient space $P\backslash SL_{2d}(\Q_p)$ of right cosets of the Siegel parabolic $P$. 
    It can be realised as the Grassmannian of maximal isotropic subspaces of a $2d$-dimensional symplectic space.
    \item[]$\schw(X)$ - The space of locally constant and compactly supported functions on a totally disconnected topological space $X$.
    \item[]$\schw=\schw(\Q_p^d)$ - The space of locally constant and compactly supported functions on $\Q_p^d$.
    \item[]$\psi$ - a non-trivial smooth character $\psi:(\Q_p,+)\map \C_p^\times$.
    \item[]$\rho_\psi$ - the Schrödinger representation of $\heis$ on $\schw$ with central character $\psi$.
    \end{itemize}
Assume that $V$ is a representation of a group $G$ over $\C_p$.
    \begin{itemize}
    \item[]$\Norms(V)^G$ - The set of norms on $V$ which are invariant under the action of $G$.
    \item[]$\Norms(V)_H^G$ - The set of homothety classes of $G$-invariant norms on $V$.
    \end{itemize}

\section{A reminder on $p$-adic Heisenberg groups and Schrödinger representations}
In this section we recall the classical theory of smooth irreducible representations of Heisenberg groups over $\Q_p$.
This section is based on \cite{kudla_notes,van1978smooth}.
Throughout this section, the following are fixed: $p$ is a prime number and $\Q_p$ is the field of $p$-adic numbers, $d\geq 1$ is an integer and $C$ is an algebraically closed field of characteristic zero.

\subsection{The Heisenberg group over $\mathbb{Q}_p$ and its smooth representations}
Let $W=\Q_p^d\oplus \Q_p^d$ and denote by $\omega$ the symplectic form on $W$ given by $\omega((x_1,y_1),(x_2,y_2))=x_1\dotprod y_2-y_1\dotprod x_2$, where $a\dotprod b$, for $a,b\in\Q_p^d$, is the standard scalar product.

We denote by $\heis=\heis_{2d+1}(\Q_p)$ the $2d+1$-dimensional Heisenberg group.
Its underlying set is $W\times \Q_p$ and the multiplication is given by 
\[[w_1,t_1]\cdot [w_2,t_2]=[w_1+w_2,t_1+t_2+\frac{1}{2}\omega(w_1,w_2)].\]
One easily verifies that the center of $\heis$, which is also its commutator subgroup, is $Z:=\braces{[0,t]\ |\ t\in \Q_p}$, and that $\heis/Z\simeq W=\Q_p^{2d}$.
In particular, $\heis$ is a two step nilpotent group.

As a topological group, $\heis$ inherits a topology from the topology of $\Q_p$.
This makes $\heis$ a totally disconnected (t.d.) and locally compact topological group.

Recall that a representation $(V,\pi)$ of a t.d. group $G$ over $C$ is said to be smooth if the stabilizer $Stab_G(v)$ in $G$ of any vector $v\in V$ is open.
A smooth representation of $G$ is called admissible if for any open compact subgroup $K\subset G$ the sub-space $V^K$ of vectors fixed by $K$ is finite dimensional.

By Schur's lemma, if $(V,\pi)$ is a smooth irreducible representation of $\heis$, the center of $\heis$ acts on $V$ via a character $\psi$, called \textit{the central character} of $\rho$.
We identity the center of $\heis$ with $\Q_p$ and view $\psi$ as a character $\psi:(\Q_p,+)\map C^\times$.
Since $\pi$ is smooth, the kernel of $\psi$ is an open subgroup of $\Q_p$ and we say that $\psi$ is a smooth character.

The classification of smooth irreducible representations of $\heis$ is well known, and we recall it.
If $\psi$ is trivial, the action of $\heis$ factors through an abelian quotient, and $V$ is $1$-dimensional.
Assume that $\psi$ is non-trivial.
We construct a representation $\rho_\psi$, called the Schrödinger representation of $\heis$, which has central character $\psi$.

Let $\schw=\schw(\Q_p^d)$ be the space of Schwartz functions, that is functions $f:\Q_p^d\map C$ which are locally constant and compactly supported.
It is an infinite dimensional vector space over $C$.
Define a representation $\rho_\psi$ of $\heis$ on $\schw$ as follows.
Let $w=(a,b)$ with $a,b\in\Q_p^d$.
Then
\[(\rho_{\psi}([w,t])f)(x)=\psi\sog{t+\frac{1}{2}a\dotprod b+b\dotprod x}\cdot f(x+a).\]

\begin{thm}[Smooth Stone-von Neumann]
Let $\psi$ be a non-trivial smooth character of $(\Q_p,+)$.
    \begin{enumerate}
    \item The representation $\rho_\psi$ is a smooth, irreducible and admissible representation of $\heis$ and has central character $\psi$.
    \item Let $(V,\pi)$ be a smooth representation of $\heis$.
    Assume that the center of $\heis$ acts via the character $\psi$.
    Then $V$ decomposes as a direct sum of sub-representations, each isomorphic to $\rho_\psi$.
    \end{enumerate}
\end{thm} 

It is important to note that in the Schrödinger representation, the Heisenberg group acts on $\schw$ by translations and by multiplication by the smooth characters of $\Q_p^d$.
A smooth character of $\Q_p^d$ is a homomorphism $\alpha:(\Q_p^d,+)\map C^\times$ with an open kernel.
By definition, any translation appears as an action of an element of the Heisenberg group. 
It is also true that if $\alpha:\Q_p^d\map C^{\times}$ is a smooth character, there exists an element $[0,b,0]$ whose action is multiplication by $\alpha$.
Indeed, if $\psi$ is a non-trivial smooth character of $\Q_p$, any smooth character of $\Q_p^d$ is of the form $\psi\circ \lambda$, for some $\lambda$ in the dual space of $\Q_p^d$.

\subsection{Automorphisms of the Heisenberg group and intertwining operators on the Schrödinger representation}
Let $J=\begin{pmatrix}0 & I_d\\ -I_d & 0 \end{pmatrix}$ and $\mathrm{Sp}_{2d}(\Q_p)$ be the symplectic group
\[\mathrm{Sp}_{2d}(\Q_p)=\braces{g\in \mathrm{GL}_{2d}(\Q_p)\ |\ gJg^t=J}.\]
Thinking about the vectors in $\Q_p^d\oplus\Q_p^d$ as row vectors, an element $g\in \mathrm{Sp}_{2d}(\Q_p)$ acts on $W$ by right multiplication: $w\mapsto wg$ and preserves the symplectic form $\omega$.
This defines a right action of the symplectic group on the Heisenberg group by automorphisms as follows:
\[[w,t]\cdot g=[wg,t].\]
These automorphisms are continuous and their restriction to the center $Z=\braces{[0,t]\ |\ t\in\Q_p}$ is the identity.
Moreover, any continuous automorphism of $\heis$ whose restriction to the center is the identity is a composition of a conjugation and an automorphism coming from the symplectic group.
    \begin{remark}
    This is not true for Heisenberg groups over extension fields of $\Q_p$, and this is the only reason why we restrict to $\Q_p$.
    \end{remark} 

Let $g\in \mathrm{Sp}_{2d}(\Q_p)$.
Define a new representation $\rho_{g,\psi}$ on $\schw$ by 
\[\rho_{g,\psi}([w,t])f=\rho_\psi([w,t]\cdot g)f\]
for any $[w,t]\in\heis$ and $f\in\schw$.
The representation $\rho_{g,\psi}$ is smooth, irreducible and has $\psi$ as its central character.
Thus, by the Stone-von Neumann theorem, $\rho_\psi\simeq \rho_{g,\psi}$, so there exists an invertible linear operator $T_g$ on $\schw$ such that 
\[ \rho_\psi([w,t])\circ T_g=T_g\circ\rho_{g,\psi}([w,t])\]
for any $[w,t]\in \heis$.
By Schur's lemma, $T_g$ is unique up to a multiplicative constant.
It also follows that $T_{g_1\cdot g_2}$ is equal, up to a constant, to $T_{g_1}\circ T_{g_2}$, so $g\mapsto T_g$ is a projective representation, called the Weil representation.
There is an explicit formula for the operators $T_g$.
Let $g\in \mathrm{Sp}_{2d}(\Q_p)$ and write it as 
\[g=\twomat{a}{b}{c}{d}\]
where $a,b,c,d\in M_d(\Q_p)$ are square matrices.
\begin{prop}[\cite{kudla_notes}, Proposition 2.3]\label{prop_intertwining_formula}
Let $g$ as above and $T_g$ an intertwining operator corresponding to $g$.
There is a unique choice of a $\C_p$-valued Haar distribution $d\mu$ on $Im(c)$ such that 
\[T_g(f)(x)=\intop_{Im(c)}\psi\sog{\frac{1}{2}(xa)\dotprod (xb)-(xb) \dotprod y+\frac{1}{2}y\dotprod (yd)}\cdot f(xa+y)\ d\mu(y).\]
Here, $Im(c)$ is the space $\braces{vc\ |\ v\in\Q_p^d}$.
\end{prop} 
Note that if $g=J$, the formula gives, up to normalization, the usual Fourier transform. 
If $c=0$, we get the operation of multiplication by a quadratic exponential accompanied by the dilation $x\mapsto xa$. 
\section{Banach representations}
The goal of this section is twofold.
We introduce the terminology about norms and Banach representations that will be used throughout this paper, and we prove some fundamental properties of Banach representations that we will later need.
We address two issues.
The first is a notion of minimality of norms that we call weak minimality.
The second consists of several characterizations of quotients of universal unitary completions of (algebraically) cyclic representations.
Let $G$ be an abstract group and $V$ a representation of $G$ over $\C_p$.
Under the assumptions that $V$ is cyclic, $V$ has a universal unitary completion in the sense of \cite{emerton2005p} that we denote by $\Ucomp{V}$.
The quotients of $\Ucomp{V}$ by closed sub-representations will play an important role in this paper, especially quotients by maximal sub-representations.
We give two intrinsic characterizations of these quotients: in terms of a special type of norms which we call \textit{locally maximal} and in terms of the existence of a special type of vectors which we call \textit{strongly cyclic}.

\subsection{General terminology and notation}
Let $V$ be vector space over $\C_p$.
A norm on $V$ is a map $\norm{}:V\map \R_{\geq 0}$ such that 
 \begin{enumerate}
    \item $\norm{v}=0$ if and only if $v=0$.
    \item $\norm{a\cdot v}=\abs{a}_p\cdot \norm{v}$ for any $v\in V$ and $a\in\C_p$.
    \item $\norm{v_1+v_2}\leq \max(\norm{v_1},\norm{v_2})$.
    \end{enumerate}
If $\norm{}$ satisfies only $2$ and $3$ we say that it is a seminorm.

Let $\norm{}_1,\norm{}_2$ be two norms on $V$.
We write $\norm{}_1\leq \norm{}_2$ if $\norm{v}_1\leq \norm{v}_2$ for any $v\in V$.
We say that $\norm{}_1$ is dominated by $\norm{}_2$, and denote it by $\norm{}_1\dominated \norm{}_2$ if there exists a constant $D>0$ such that $\norm{}_1\leq D\cdot \norm{}_2$.
We say that $\norm{}_1$ and $\norm{}_2$ are equivalent if each dominates the other: $\norm{}_1\dominated \norm{}_2$ and $\norm{}_2\dominated \norm{}_1$.
These two norms are called homothetic if there exists $c>0$ such that $\norm{v}_1=c\cdot \norm{v}_2$ for any $v\in V$. 

If $v\in V$ is a non-zero vector, we say that $\norm{}$ is normalized at $v$ if $\norm{v}=1$.
In any homothety class of norms there is exactly one norm that is normalized at $v$.

Given a norm $\norm{}$ on $V$, we denote the completion of $V$ with respect to $\norm{}$ by $\completion{V}{\norm{}}$.

Assume that a group $G$ acts on $V$.    
A norm $\norm{}$ on $V$ is said to be $G$-invariant if $\norm{gv}=\norm{v}$ for any $v\in V$ and $g\in G$.    
When there is no ambiguity about the group $G$, we will simply say that $\norm{}$ is an invariant norm.
We denote the set of norms on $V$ by $\Norms(V)$ and by $\Norms(V)^G$ its subset of $G$-invariant norms.

In this paper the term \textit{Banach representation} means the following.
\begin{defn}\label{def_Banach_rep}
A Banach representation (over $\C_p$) of $G$ is a pair $(B,\norm{})$ of a $G$-representation $B$ and a $G$-invariant norm $\norm{}$ such that $B$ is complete with respect to $\norm{}$.
\end{defn} 

A morphism of Banach representations of $G$ is a continuous $G$-equivariant map, but it need not be an isometry.
In particular, isomorphic Banach representations of $G$ are not necessarily isometric.

\subsection{Weakly minimal norms}
\begin{defn}
Let $(B,\norm{})$ be a Banach representation of the group $G$ and $v\in B$ a non-zero vector.
We say that $\norm{}$ is weakly minimal at $v$ if the following holds.
    \begin{itemize}
    \item For any $G$-invariant norm $\norm{}'$ on $B$ such that $\norm{}'\leq \norm{}$ and $\norm{v}'=\norm{v}$, we have $\norm{}'=\norm{}$.
    \end{itemize}
\end{defn} 

\begin{lem}
Let $(B,\norm{})$ be a Banach representation of $G$.
Assume that the values of $\norm{}$ are the same as the values of $\abs{}_p$ on $\C_p$.
Assume that $v\in B$ is a non-zero vector such that $\norm{v}=1$ and such that its image $\overline{v}$ in the quotient 
\[\overline{B}(\norm{}):=\quot{\braces{v\in B\ |\ \norm{v}\leq 1}}{\braces{v\in B\ |\ \norm{v}< 1}}\]
is contained in any non-zero sub-representation of $\overline{B}(\norm{})$.
Then $\norm{}$ is weakly minimal at $v$.
\end{lem} 
	\begin{proof}
    Let $\norm{}'\in \Norms(B)^G$ be a $G$-invariant norm such that $\norm{v}'=1$ and $\norm{}'\leq \norm{}$.
    The identity map $Id:B\map B$ induces a map $T:\overline{B}(\norm{})\map \overline{B}(\norm{}')$.
    The kernel of $T$ is a sub-representation of $\overline{B}(\norm{})$ that does not contain $v$, hence by assumption, this kernel is trivial.
    It follows that $T$ is injective.
    Therefore, $\norm{w}'=1$ for any $w$ with $\norm{w}=1$.
    Since the values of $\norm{}$ are the same as the values of $\abs{}_p$, $\norm{}'=\norm{}$.
    \end{proof} 
    
\begin{prop}\label{prop_weak_minimality}
Let $G$ be a pro-$p$ group.
Let $C(G)$ denote the space of continuous functions on $G$ with values in $\C_p$ and let $\supnorm{}$ be the sup norm on $C(G)$. 
Consider the action of $G$ on $C(G)$ by right translations.
The sup norm is weakly minimal at $\textbf{1}$, where $\textbf{1}$ denotes the constant function $\textbf{1}(x)=1$.
\end{prop} 
    \begin{proof}
    Identify the quotient
    \[\quot{\braces{f\in C(G)\ |\ \supnorm{f}\leq 1}}{\braces{f\in C(G)\ |\ \supnorm{f}< 1}}\]
    with the space $\schw(G,\overline{\fld{F}}_p)$ of locally constant functions on $G$ with values in an algebraic closure $\overline{\fld{F}}_p$ of $\fld{F}_p$.
    By the previous lemma, it is enough to show that any non-zero sub-representation of $\schw(G,\overline{\fld{F}}_p)$ contains the constant function $\textbf{1}$.
    Let $f\in \schw(G,\overline{\fld{F}}_p)$ be non-zero and denote by $V$ the sub-representation generated by $f$.
    As $f$ is fixed by some open normal subgroup $N\subset G$, $V$ is a cyclic representation of the finite group $G/N$.
    In particular, $V$ is a finite dimensional representation of the finite $p$-group $G/N$ over $\overline{\fld{F}}_p$.
    Thus, $V$ contains a non-zero $G$-invariant vector $\phi$.
    This $\phi$ is a non-zero constant function.
    \end{proof} 
\subsection{The universal unitary completion of a cyclic representation}
Let $V$ be a representation of $G$ and assume that $v\in V$ is a cyclic vector.
In addition, assume that $\Norms(V)^G$ is non-empty, i.e. there exists a $G$-invariant norm on $V$.

If $\norm{}\in \Norms(V)^G$, its closed unit ball $\braces{w\in V\ |\ \norm{w}\leq 1}$ is an $\mathcal{O}_{\C_p}[G]$-module that contains a non-zero multiple of any vector in $V$, but contains no $\C_p$-lines.
Such an $\mathcal{O}_{\C_p}$-module is called an integral structure.
Conversely, any integral structure $L$ defines a $G$-invariant norm, called the gauge of $L$, by
\[\norm{v}_L=\inf\braces{\abs{a}_p\ |\ v\in a\cdot L}.\]
We stress the fact that in general $L$ might not be equal to the closed unit ball nor to the open unit ball of $\norm{}_L$, but lies strictly between them.
 For future use we record the following formulas for the closed and open unit balls of $\norm{}_L$,
    \begin{equation}\label{equation_close_open_unit_balls}
    \braces{v\in V\ |\ \norm{v}_L\leq 1}=\bigcap_{\substack{\lambda\in \C_p\\ \pabs{\lambda}>1}}\lambda L,\ \ \ 
    \braces{v\in V\ |\ \norm{v}_L< 1}=\bigcup_{\substack{\lambda\in \C_p\\ \pabs{\lambda}<1}}\lambda L.  
    \end{equation} 

Since $\C_p$ is not discretely valued, two different invariant norms give rise to different integral structures, but two different integral structures might define the same norm.
Nevertheless, the correspondence between invariant norms and integral structures inverts order.

The set $L_{v}:=\mathcal{O}_{\C_p}[G]\cdot v$ is an integral structure.
Indeed, it contains a multiple of any vector since $v$ is cyclic, and it contains no lines because of the existence of an invariant norm.
As $L_v$ is the smallest integral structure that contains $v$,
its corresponding norm, which we denote by $\norm{}_v$, is normalized at $v$ and is the maximal invariant norm normalized at $v$.
This means that if $\norm{}\in\Norms(V)^G$ is normalized at $v$, then $\norm{}\leq \norm{}_v$.
We call the norm $\norm{}_{v}$ \textit{the maximal invariant norm at $v$} or \textit{the maximal norm at $v$} for short.

If $v_1,v_2\in V$ are two cyclic vectors of $V$, the norms $\norm{}_{v_1}$ and $\norm{}_{v_2}$ are equivalent.
In particular the completion of $V$ with respect to $\norm{}_v$, where $v$ is a cyclic vector, is independent of $v$ as a topological vector space.
We denote this completion by $\Ucomp{V}$ and call it the universal unitary completion of $V$, or the universal completion for short. 
Note that this is a particular case of Example A in \cite{emerton2005p}.
The universal completion of $V$ has the following universal property: if $(B,\norm{})$ is a Banach representation of $G$ and $T:V\map (B,\norm{})$ is  $G$-equivariant, then $T$ factors continuously through $\Ucomp{V}$.

    \begin{remark}
    The assumption that $\Norms(V)^G\neq \phi$ is superfluous, is made for simplification and because this is the case that will appear later.
    If $V$ does not have a $G$-invariant norm, $\mathcal{O}_{\C_p}[G]\cdot v$ contains $\C_p$-lines.
    The union of these lines is a sub-representation $W\subset V$ and the quotient $V'=V/W$ is cyclic and has an invariant norm.
    The universal completion of $V$ is $\Ucomp{V'}$.    
    \end{remark} 

Any element of $\Ucomp{V}$ can be written as 
\[\sum_{g\in G}\lambda_g\cdot g(v)\]
where $(\lambda_g)_{g\in G}\subset \C_p$ is \textit{summable}, meaning that for any $\eps>0$, at most finitely many of the $\lambda_g$ satisfy $\abs{\lambda_g}_p\geq \eps$.
Conversely, any element of this form is in $\Ucomp{V}$.

\begin{defn}\label{def_local_maximality}
Let $W$ be a representation of $G$ and $\norm{}\in\Norms(W)^G$.
Let $w\in W$ a non-zero vector such that $\norm{}$ is normalized at $w$.
We say that $\norm{}$ is \textit{locally maximal at $w$} if the following property holds.
    \begin{itemize}
    \item For any $\norm{}'\in \Norms(W)^G$ that is normalized at $w$ and is dominated by $\norm{}$ we have $\norm{}'\leq \norm{}$.
    \end{itemize}
\end{defn} 
For example, the norm $\norm{}_v$ on $V$ is a locally maximal norm at $v$.
Were it also locally maximal at another cyclic vector $u$, then the two norms $\norm{}_v$ and $\norm{}_u$, being equivalent, would be homothetic.
Easy examples show that this need not be the case.
Thus a norm which is locally maximal at one cyclic vector is in general not locally maximal at another one.
For another example, consider the space $C_0(\Q_p^d)$ of $\C_p$-valued continuous functions on $\Q_p^d$ that go to zero at infinity, and the action of the Heisenberg group $\heis$ on it by the formula given in the previous section.
We will later show that the sup norm $\supnorm{}$ is a locally maximal norm at $\textbf{1}_{\Z_p^d}(x)$ on $C_0(\Q_p^d)$, where $\textbf{1}_{\Z_p^d}$ is the characteristic function of $\Z_p^d$.

\begin{defn}\label{def_srongly_cyclic_vector}
Let $(B,\norm{})$ be a Banach representation of $G$ and $v\in B$ a non-zero vector.
We say that $v$ is \textit{topologically cyclic} if $v$ generates (algebraically) a dense representation in $B$.
We say that $v$ is \textit{strongly cyclic} if any $w\in V$ can be written as
\[w=\sum_{g\in G}\lambda_g\cdot g(v),\]
where the $(\lambda_g)_{g\in G}$ is summable.
\end{defn} 
For example, if $v$ is a cyclic vector in $V$, $v$ is strongly cyclic in $\Ucomp{V}$.
In the end of this section we give an example of a topologically cyclic vector which is not a strongly cyclic vector.

We make the following two observations.
Let $T:B'\map B$ be a map of Banach representations of $G$.
    \begin{itemize}
    \item Assume that the image of $T$ contains a strongly cyclic vector. 
    Then $T$ is surjective.
    \item Assume that $v'\in B'$ is strongly cyclic and that $T$ is surjective.
    Then $v=T(v')$ is strongly cyclic in $B$.
    \end{itemize}

We begin with two lemmas.
The first says that a quotient norm of a locally maximal one is locally maximal.
The second says that strongly cyclic vectors give rise to locally maximal norms.

\begin{lem}\label{lem_quotient_locally_maximal_norms}
Let $W$ be a representation of $G$, $\norm{}\in \Norms(W)^G$ and $K\subset W$ a closed (with respect to $\norm{}$) sub-representation.
Assume that $\norm{}$ is normalized and locally maximal at $w\in W$.
Then the quotient norm on $W/K$ is normalized and locally maximal at the image of $w$.
\end{lem} 
    \begin{proof}
    Let $\norm{}'$ denote the quotient norm on $W/K$.
    Let $\norm{}_2$ be a $G$-invariant norm on $W/K$ that is normalized at the image of $w$ and dominated by $\norm{}'$.
    Using the quotient map $W\map W/K$ we view $\norm{}_2$ also as a semi-norm on $W$.
    Then $\max(\norm{}_2,\norm{})$ is a $G$-invariant norm on $W$ that is normalized at $w$ and dominated by $\norm{}$.
    Thus, $\norm{}_2\leq \norm{}$, as semi-norms on $W$.
    Taking the quotient by $K$, we obtain $\norm{}_2\leq \norm{}'$, as norms on $W/K$.    
    \end{proof}

\begin{lem}\label{lem_locally_maximal_from_strongly_cyclic}
Assume that $(B,\norm{})$ is a Banach representation of $G$ and that $0\neq v\in B$ is a strongly cyclic vector.
Then there exists a unique norm, which we denote by $\norm{}_{v,B}$, which is normalized and locally maximal at $v$ and is equivalent to $\norm{}$.
In addition, if $w\in B$ with $\norm{w}_{v,B}=r$, then for any $\eps>0$ there exists a summable sequence $(\lambda_g)_{g\in G}$ such that 
\[w=\sum_{g\in G}\lambda_g\cdot g(v)\]
and $\max_{g\in G}\pabs{\lambda_g}<(1+\eps)\cdot r$.
\end{lem} 

    \begin{proof}
    The uniqueness of a normalized and locally maximal norm at $v$ is clear.
    In the rest of the proof we construct the norm $\norm{}_{v,B}$ using an integral structure and show the additional property.
    
    Let $\overline{L}$ be the closure in $B$ of 
    \[L=\braces*{\sum_{g\in G}\lambda_g\cdot g(v)\ |\ (\lambda_g)_{g\in G} \text{ is summable and } \abs{\lambda_g}_p\leq 1 \text{ for all }g\in G}.\]
    Assume, for convenience, that $\norm{v}=1$. 
    We first show that $\overline{L}$ is an open integral structure.
    It is straightforward that $\overline{L}$ is an integral structure, the only non-obvious part is that $\overline{L}$ contains no $\C_p$-lines.
    This is true since $L$, and therefore $\overline{L}$, is contained in the closed unit ball of $\norm{}$.
    We now show that $\overline{L}$ is open.
    Since $v$ is strongly cyclic, $B=\bigcup_{n=0}^\infty p^{-n}\cdot \overline{L}$.
    Since $B$ is a complete metric space, it follows from Baire's category theorem that $\overline{L}$ has a non-empty interior, and since it is a topological subgroup, it must be open.
    Let $\norm{}_{v,B}$ be the norm that corresponds to $\overline{L}$.
    Since $\overline{L}$ is an open integral structure, $\norm{}_{v,B}$ is a $G$-invariant norm and equivalent to $\norm{}$.
    
    Now we show that $\norm{}_{v,B}$ is normalized at $v$ and locally maximal at $v$.
    Since $v\in \overline{L}$, it follows that $\norm{v}_{v,B}\leq 1$.
    Let $\norm{}'$ be a $G$-invariant norm dominated by $\norm{}_{v,B}$ and normalized at $v$.
    The closed unit ball of $\norm{}'$ contains $L$, and since
    \[\norm{}'\dominated\norm{}_{v,B}\dominated \norm{}\]
    its unit ball is closed in $B$.
    Thus, the unit ball of $\norm{}'$ contains $\overline{L}$.
    Therefore, $\norm{}'\leq \norm{}_{v,B}$.
    Substituting $v$, we see that $\norm{v}_{v,B}\geq 1$, so $\norm{}_{v,B}$ is normalized at $v$ and locally maximal at $v$.
    
    Finally, we prove the additional property.
    Let $w\in B$ and let $\eps>0$.
    We may assume that $(1+\eps)^{-1}<\norm{w}_{v,B}<1$.
    By \ref{equation_close_open_unit_balls} (formula for the open unit ball), $w\in\overline{L}$, so there exists $w_0\in L$ such that $\norm{w-w_0}_{v,B}<p^{-1}$.
    Similarly, there exists $w_1\in pL$ such that $\norm{w-w_0-p\cdot w_1}_{v,B}<p^{-2}$.
    Continuing in this manner we obtain a sequence $(w_n)_{n=0}^\infty$, where $w_n\in p^n\cdot L$ for all $n$, and $w=\sum_{n=0}^\infty w_n$.
    Therefore, $w\in L$, so it can be written as $w=\sum_{g\in G}\lambda_g\cdot g(v)$ and $\max_{g\in G}\pabs{\lambda_g}<1<(1+\eps)\cdot \norm{w}_{v,B}$.
    \end{proof} 

    \begin{remark}
    Note that the last step in the proof can be modified slightly to show that $L=\overline{L}$.
    However, the closed unit ball of $\norm{}_{v,B}$ might be strictly larger than $L$.
    \end{remark}

\begin{thm}\label{thm_strongly_cyclic_spaces}
Let $(B,\norm{})$ be a Banach representation of $G$ and $v\in B$ a non-zero vector.
The following are equivalent.
    \begin{enumerate}
    \item $v$ is a strongly cyclic vector of $B$ and $\norm{}=\norm{}_{v,B}$.
    \item Let $V$ be the (algebraic) sub-representation generated by $v$.
    Then the map $I:\Ucomp{V}\map B$ is surjective and if $K$ is its kernel, the induced map $\Ucomp{V}/K\map B$ is an isometry when we equip $\Ucomp{V}$ with the norm $\norm{}_v$.
    \item $\norm{}$ is normalized at $v$ and locally maximal at $v$.
    \end{enumerate}
\end{thm} 
    \begin{proof}
    We will show $(1)\Rightarrow (2)\Rightarrow (3)\Rightarrow (1)$.
    Assume $(1)$. 
    Since $v$ is strongly cyclic, the map $I:\Ucomp{V}\map B$ is surjective. 
    Equip $\Ucomp{V}$ with the norm $\norm{}_v$ and let $\norm{}'$ denote the quotient norm on $\Ucomp{V}/K$.
    Via $I$, we view $\norm{}'$ as a norm on $B$.
    By the open mapping theorem, $\norm{}'$ and $\norm{}=\norm{}_{v,B}$ are equivalent.
    By Lemmas \ref{lem_quotient_locally_maximal_norms} and \ref{lem_locally_maximal_from_strongly_cyclic}, these two norms are normalized and locally maximal at $v$.
    Thus, they are equal.
    
    Assume $(2)$.
    $(3)$ Follows from Lemma \ref{lem_quotient_locally_maximal_norms}.
    
    Assume $(3)$.
    We will prove $(1)$.
    We assume that $\norm{}$ is normalized and locally maximal at $v$.
    Let $w\in B$; we want to show that $w$ is of the form $\sum_{g\in G}\lambda_g\cdot g(v)$, where $(\lambda_g)_{g\in G}$ is summable.
    We may assume that $\norm{w}\leq 1$.
    Let $L=\mathcal{O}_{\C_p}[G]\cdot v$ and let $D$ be the closed unit ball of $\norm{}$.
    Then $L+p^2\cdot D$ is an open integral structure in $B$ that contains $v$.
    Its corresponding norm, that we denote by $\norm{}'$, is dominated by $\norm{}$ and satisfies $\norm{v}'\leq 1$.
    Therefore, by $(3)$, $\norm{}'\leq \norm{}$.
    By \ref{equation_close_open_unit_balls} (formula for the closed unit ball) it follows that 
    \[w\in D\subset\bigcap_{\substack{\lambda\in\C_p}{\pabs{\lambda}>1}}\lambda(L+p^2\cdot D)\subset p^{-1}(L+p^2\cdot D)=p^{-1}L+pD.\]
    Thus, there exist $x_1\in L$ and $d_1\in D$ such that $w=p^{-1}x_1+p\cdot d_1$.
    Repeating this process with $d_1$ instead of $w$, there exist $x_2\in L$ and $d_2\in D$ such that $d_1=p^{-1}x_2+pd_2$.
    Thus, $w=p^{-1}x_1+p(p^{-1}x_2+pd_2)=p^{-1}x_1+x_2+p^2d_2$.
    Repeating this process, we obtain sequences $(x_n)_{n=1}^\infty\subset L$ and $(d_n)_{n=1}^\infty \subset D$ such that $d_n=p^{-1}x_{n+1}+pd_{n+1}$ for any $n\geq 0$.
    Thus, $w=\sum_{n=0}^\infty p^{n-1}\cdot x_n$ which is of the desired form.    
    \end{proof} 

In particular, if $(B,\norm{})$ is a Banach representation of $G$, the following are equivalent.
    \begin{enumerate}
    \item $B$ is isomorphic to a quotient of a universal completion of a cyclic representation.
    \item $B$ has a strongly cyclic vector.
    \item $\norm{}$ is equivalent to a locally maximal norm with respect to some vector $v$.
    \item There exists $v\in B$ such that any map $T:B'\map B$ of Banach representations of $G$ such that $v$ lies in its image is surjective.
    \end{enumerate}

\subsection{Strongly irreducible Banach representations}

\begin{defn}\label{def_strong_irreducible}
Let $(B,\norm{})$ be a Banach representation of $G$.
We say that $B$ is strongly irreducible if any non-zero vector in $B$ is strongly cyclic. 
\end{defn}
Clearly, a strongly irreducible Banach representation is topologically irreducible.
The converse is not true (see the example at the end of this section).

\begin{prop}\label{prop_open_unit_ball_of_strongly_cyclic}
Let $(B,\norm{})$ be a Banach representation of $G$ and $v\in B$ a strongly cyclic vector.
Assume that $\norm{}$ is normalized and locally maximal at $v$.
Then, any $w\in B$ with $\norm{v-w}<1$ is also strongly cyclic.
\end{prop} 
    \begin{proof}
    First, note that $\norm{}$ is equal to $\norm{}_{v,B}$ from Lemma \ref{lem_locally_maximal_from_strongly_cyclic}.
    Let $w\in B$ such that $\norm{v-w}<1$.
    By Theorem \ref{thm_strongly_cyclic_spaces}, it is enough to show that $\norm{}$ is normalized and locally maximal at $w$.
    That $\norm{}$ is normalized at $w$ follows from the strong triangle inequality.
    To show that $\norm{}$ is locally maximal at $w$, let $\norm{}'\in\Norms(V)^G$ a norm that is dominated by $\norm{}$ and normalized at $w$.
    By Lemma \ref{lem_locally_maximal_from_strongly_cyclic} we can write $v-w=\sum_{g\in G}\lambda_g\cdot g(v)$, where $(\lambda_g)_{g\in G}$ is summable and $\max_{g\in G}(\lambda_g)<1$.
    Therefore, if $\norm{w}'=1$, then also $\norm{v}'=1$.
    Since $\norm{}$ is locally maximal at $v$, $\norm{}'\leq \norm{}$.
    This prove that $\norm{}$ is also locally maximal at $w$.
    \end{proof} 

\begin{prop}
Let $(B,\norm{})$ be a Banach representation of $G$.
The set of all strongly cyclic vectors in $B$ is open (possibly empty).
\end{prop} 
    \begin{proof}
    Assume that the set of strongly cyclic vectors in $B$ is not empty.
    Let $v$ be a strongly cyclic vector in $B$.
    By the previous proposition, all the vectors in the open unit ball around $v$ with respect to $\norm{}_{v,B}$ are strongly cyclic.
    By Lemma \ref{lem_locally_maximal_from_strongly_cyclic}, the open unit ball of $\norm{}_{v,B}$ is open in $B$.
    \end{proof} 
    
\begin{cor}\label{cor_proper_subspaces}
Let $(B,\norm{})$ be a Banach representation of $G$ and assume that $B$ contains a strongly cyclic vector.
If $B$ is not strongly irreducible, $B$ contains a non-zero proper closed sub-representation.
\end{cor} 
    \begin{proof}
    Let $U$ be the open subset of strongly cyclic vectors in $B$.
    By assumption, $U$ is not empty.
    Assume that $B$ is not strongly irreducible and let $0\neq w\in B$ a non strongly cyclic vectors.
    Denote by $W$ the algebraic representation generated by $w$ in $B$, then $W\subset B\backslash U$.
    Thus, the closure of $W$ is a proper non-zero closed sub-representation of $B$.
    \end{proof}

\begin{thm}\label{Thm_strong_irreducibility}
Let $(B,\norm{})$ be a Banach representation of $G$.
The following are equivalent.
    \begin{enumerate}
    \item $B$ is strongly irreducible.
    \item $B$ is topologically irreducible and there exists a strongly cyclic vector in $B$.
    \item $B$ is topologically irreducible and there exists a locally maximal norm at some vector $0\neq v\in B$ on $B$, equivalent to $\norm{}$.
    \item Any non-zero $G$-equivariant bounded map $B'\map B$, where $(B',\norm{}')$ is a Banach representation of $G$, is surjective.
    \item $(B,\norm{})$ is isomorphic to a quotient of a universal completion of a cyclic representation of $G$ by a maximal sub-representation.
    \end{enumerate}
\end{thm} 
    \begin{proof}
    The implication $(1)\Rightarrow (2)$ is trivial and the implication $(2)\Rightarrow(1)$ follows from Corollary \ref{cor_proper_subspaces}.
    The equivalence $(2)\iff (3)$ follows from Theorem \ref{thm_strongly_cyclic_spaces}.
    Next we show $(1)\iff(4)$.
    Assume $(1)$. 
    Let $(B',\norm{}')$ be a Banach representation of $G$ and $T:B'\map B$ a non-zero $G$-equivariant bounded map.
    Let $0\neq v\in Im(T)$.
    Then $v$ is strongly cyclic and by a previous observation, $T$ is surjective.
    Assume $(4)$.
    Let $v\in B$ non-zero.
    Let $V$ be the algebraic representation generated by $v$ in $B$.
    The map $I:\Ucomp{V}\map B$ is a non-zero $G$-equivariant bounded map, so by assumption $I$ is surjective.
    Thus, by a previous observation, $v$ is strongly cyclic in $B$.
    Finally, we prove $(2)\iff (5)$.
    Assume $(2)$.
    Let $v\in B$ be a strongly cyclic vector, let $V$ be the algebraic representation generated by $v$ and $I:\Ucomp{V}\map B$.
    Since $v$ is strongly cyclic, $I$ is surjective.
    If $W\subset\Ucomp{V}$ denotes the kernel of $I$, then $\Ucomp{V}/W$ is isomorphic to $B$.
    Since $B$ is topologically irreducible, $W$ is a maximal closed sub-representation.
    Assume $(5)$.
    As a quotient of a universal completion of a cyclic representation by a maximal sub-representation, $B$ is topologically irreducible and contains a strongly cyclic vector. 
    By Corollary \ref{cor_proper_subspaces}, $B$ is strongly cyclic.
    \end{proof} 

\begin{prop}\label{prop_small_representations}
Let $(B_i,\norm{}_i)$, for $i=1,...,n$, be pairwise non-isomorphic strongly irreducible Banach representations of $G$. 
Let $B$ be the Banach representation $B=\bigoplus_{i=1}^nB_i$, equipped with the norm $\max\sog{\norm{}_1,...,\norm{}_n}$.
Then any $x=(x_1,...,x_n)\in B$ such that $x_i\neq 0$ for all $i$ is strongly cyclic in $B$.
\end{prop} 
    \begin{proof}
    By induction on $n$.
    The case $n=1$ is trivial.
    Assume that $n>1$ and that the claim holds for any $1\leq k<n$.
    By Theorem \ref{thm_strongly_cyclic_spaces}, it is enough to show that for any Banach representation $(B',\norm{}')$ of $G$ and any continuous map of representations $T:B'\map B$, if $x$ lies in the image of $T$ then $T$ is surjective.
    Let $T:B'\map B$ be such a map.
    Denote by $P_1:B\map B_1$ and $P_2:B\map \bigoplus_{i=2}^nB_i$ the projections.
    Since $x_1$ lies in the image of $P_1\circ T$ and $(x_2,...,x_n)$ lies in the image of $P_2\circ T$, it follows from the induction hypothesis that both $P_1\circ T$ and $P_2\circ T$ are surjective.
    Let $K_1,K_2$ be the kernels of $P_1\circ T$ and $P_2\circ T$ respectively.
    Then $K_2$ is not contained in $K_1$, for otherwise we would have a non-zero map $\bigoplus_{i=2}^nB_i\map B_1$.
    Such a map would give a non-zero map between one of the $B_i$, for $i\geq 2$, and $B_1$.
    Since both $B_1$ and $B_i$ are strongly irreducible, such a map must be an isomorphism, contradicting the hypothesis.
    Therefore, the restriction of $P_1\circ T$ to $K_2$ is a non-zero map, hence surjective since $B_1$ is strongly irreducible.
    It follows that $B_1$ is contained in the image of $T$.
    Similarly, for any $1\leq i\leq n$, $B_i$ is contained in the image of $T$.
    Thus, $T$ is surjective.
    \end{proof} 
    
We end this section with an example of a topologically irreducible Banach representation which is not strongly irreducible.
    \begin{exmp}
    Let $C(\Z_p)$ be the space of continuous functions on $\Z_p$ with values in $\C_p$, equipped with the sup norm $\supnorm{}$.
    We choose $\zeta\in \C_p$, not a root of unity, such that $\abs{\zeta-1}_p<1$ and denote by $G$ the group generated by translations and by multiplications by $\zeta^{nx}$, $n\in\Z$.
    The sup norm is invariant under the action of $G$, so $C(\Z_p)$ is a Banach representation of $G$.
    
    As a representation of $G$, $C(\Z_p)$ is topologically irreducible, as we now show.
    Let $A$ be the linear span of the functions $\zeta^{nx}$, $n\in\Z$.
    Then $A$ is an algebra in $C(\Z_p)$ that separates points and contains the constant functions.
    By the $p$-adic Stone-Weierstrass theorem (\cite{kaplansky1995weierstrass}), $A$ is dense in $C(\Z_p)$.
    Let $0\neq f\in C(\Z_p)$ and $W$ be the closed sub-representation generated by $f$.
    We will show that $W=C(\Z_p)$.
    Applying translations, $W$ contains a nowhere vanishing functions $g(x)$.
    Then $A\cdot g\subset W$ is dense in $C(\Z_p)$, so $W=C(\Z_p)$.
    
    We show that the constant function $\textbf{1}(x)$ is not a strongly cyclic vector, thus $C(\Z_p)$ is not strongly irreducible.
    Let $f(x)\in C(\Z_p)$ and assume that it can be written as
    \[f(x)=\sum_{n\in\Z}\lambda_n\cdot \zeta^{nx},\]
    where $\limit{n}\lambda_n=0$.
    The Mahler expansion of $f(x)$ is 
        \begin{align*}
        f(x)=
        \sum_{n\in\Z}\lambda_n\cdot (\zeta^{n})^x
        =\sum_{n\in\Z}\lambda_n\cdot \sum_{k=0}^\infty (\zeta^n-1)^k\cdot\binom{x}{k}
        =\sum_{k=0}^\infty \sog{\sum_{n\in\Z}\lambda_n\cdot (\zeta^n-1)^k}\cdot\binom{x}{k}
        =\sum_{k=0}^\infty b_k\cdot\binom{x}{k}.
        \end{align*} 
    There exists $0<\eps<1$ such that $\pabs{\zeta^n-1}<\eps$ for all $n\in\Z$.
    Let $m=\max_{n\in\Z}\pabs{\lambda_n}$.
    Then the coefficients $(b_k)_{k=1}^\infty$ obey the asymptotic formula
    \[\pabs{b_k}\leq m\cdot \eps^k.\]
    
    In particular, the function $f\in C(\Z_p)$ with Mahler expansion $f(x)=\sum_{k=0}^\infty p^k\cdot\binom{x}{p^k}$ is not of the form $\sum_{n\in\Z}\lambda_n\cdot\zeta^{nx}$.
      
    \end{exmp} 
\section{The main results}
In this section and for the rest of this paper all the representations are assumed to be over $\C_p$.
Fix a non-trivial smooth character $\psi:(\Q_p,+)\map \C_p^\times$, and let $(\rho_\psi,\schw)$ be the Schrödinger representation of the Heisenberg group $\heis=\heis_{2d+1}(\Q_p)$.
In particular, the functions in $\schw=\schw(\Q_p^d)$ are valued in $\C_p$.
The action of $\heis$ on $\schw$ is generated by translations and multiplication by smooth characters.
An $\heis$-invariant norm on $\schw$ is therefore a norm $\norm{}$ on $\schw$ such that
\[\norm{f(x+a)}=\norm{f(x)},\ \ \ \norm{\chi(x)\cdot f(x)}=\norm{f(x)}\]
 for any $f\in\schw$, any $a\in\Q_p$ and any smooth character $\chi$ of $\Q_p^d$.

Our main results concern a family of $\heis$-invariant norms on $\schw$ with a surprising rigidity.
This family is the orbit of the sup norm by intertwining operators.
In the first sub-section we define these norms and show that they are parameterized by a Grassmannian.
In the second sub-section we state the main results.
The proofs are given in the next sections.

\subsection{A special family of $\mathcal{H}$-invariant norms parameterized by a Grassmannian}
Let $g=\twomat{a}{b}{c}{d}$ be a matrix in the symplectic group $\mathrm{Sp}_{2d}(\Q_p)$ and choose $T_g$, a corresponding intertwining operator. 
If $\norm{}$ is an $\heis$-invariant norm on $\schw$, the norm $f\mapsto \norm{T_g(f)}$ is also $\heis$-invariant.
Indeed, 
\[\norm{T_g([w,t]f)}=\norm{[wg,t]T_g(f)}=\norm{T_g(f)}.\]
As the $T_g$ are determined up to a constant, this defines a right action of $\mathrm{Sp}_{2d}(\Q_p)$ on the space $\invHomothety$ of homothety classes of $\heis$-invariant norm on $\schw$.
If $x\in\invHomothety$ denotes the homothety class of the norm $\norm{}$, then we denote by $xg$ the homothety class of the norm $\norm{T_g(\cdot )}$.
Then $(xg_1)g_2=x(g_1g_2)$ are both equal to the homothety class of the norm $\norm{T_{g_1}(T_{g_2}(\cdot))}$.

An important example of an $\heis$-invariant norm on $\schw$ is the sup norm:
\[\supnorm{f}=\sup_{x\in \Q_p^d}\ \abs{f(x)}_p.\]
In the following proposition we determine the stabilizer in $\mathrm{Sp}_{2d}(\Q_p)$ of the homothety class of the sup norm.

\begin{prop}\label{prop_norms_Grassmannian}
Let $g=\twomat{a}{b}{c}{d}$ be a matrix in the symplectic group $\mathrm{Sp}_{2d}(\Q_p)$ and $T_g$ a corresponding intertwining operator.
The norms $\supnorm{}$ and $\supnorm{T_g(\cdot)}$ are homothetic if and only if they are equivalent, if and only if $c=0$.
\end{prop} 
    \begin{proof}
    If $c=0$, Proposition \ref{prop_intertwining_formula} says that there exists $\lambda\in\C_p^\times$ such that
    \[T_g(f)(x)=\lambda\cdot \psi\sog{\frac{1}{2}(xa)\dotprod(xb)}\cdot f(xa).\]
    As $a$ must be invertible, $\supnorm{f(xa)}=\abs{\lambda}_p\cdot\supnorm{f(x)}$.
    Thus, $\supnorm{}$ and $\supnorm{T_g(\cdot)}$ are homothetic and therefore equivalent.
    
    Assume that $c\neq 0$ and let $k\geq 1$ be the dimension of $Im(c)$.
    Recall that $c$ acts on $\Q_p^d$ by $v\mapsto v\cdot c$.
    Choose a basis $v_1,..,v_k$ of $Im(c)$ and complete it to a basis $v_1,...,v_k,v_{k+1},...,v_d$ of $\Q_p^d$.
    Let $U_n$ and $V_n$ be the compact open sets in $\Q_p^d$ and in $Im(c)$ respectively, given by
    \[U_n=\braces{\sum_{i=1}^d\lambda_iv_i\ |\ \lambda_1,...,\lambda_d\in p^n\Z_p},\ \ \ 
    V_n=\braces{\sum_{i=1}^k\lambda_iv_i\ |\ \lambda_1,...,\lambda_k\in p^n\Z_p}.\]
    Denote by $f_n(x)$ the characteristic function of $U_n$.
    Note that $f_n(x)\in\schw$.
    By Proposition \ref{prop_intertwining_formula}, there exists a Haar distribution $d\mu$ on $Im(c)$ such that
    \[T_g(f_n)(x)=\intop_{Im(c)}\psi\sog{\frac{1}{2}(xa)\dotprod(xb)-(xb)\dotprod y+\frac{1}{2}y\dotprod(yd)}\cdot f_n(xa+y)\ d\mu(y).\]
    We may assume that $\mu(V_0)=1$.
    Substituting $x=0$, we obtain
        \begin{align*}
        T_g(f_n)(0)
        =\intop_{Im(c)}\psi\sog{\frac{1}{2}y\dotprod(yd)}\cdot f_n(y)\ d\mu(y)
        =\intop_{V_n}\psi\sog{\frac{1}{2}y\dotprod(yd)}\ d\mu(y).
        \end{align*} 
    When $n$ is sufficiently large, $\frac{1}{2}y\dotprod (yd)\in \ker(\psi)$ for any $y\in V_n$, so
    \[T_g(f_n)(0)=\intop_{V_n}1\ d\mu(y)=p^{-nk}.\]
    Thus, $\limit{n}\supnorm{T_g(f_n)}=\infty$, whereas $\supnorm{f_n}=1$ for any $n$.
    Then $\norm{}_g$ and $\supnorm{}$ are not equivalent and therefore not homothetic.
    \end{proof} 

Let $P$ be the Siegel parabolic subgroup 
\[P=\braces*{\twomat{a}{b}{0}{d}\in \mathrm{Sp}_{2d}(\Q_p)},\]
and denote $\Grassmannian=P\backslash \mathrm{Sp}_{2d}(\Q_p)$.
Then $\Grassmannian$ is the Grassmannian of maximal isotropic subspaces of $(W,\omega)$.
\begin{defn}
We denote the point that corresponds to $P$ in $\Grassmannian$ by $\infty$.
For any $\alpha=Pg\in \Grassmannian$ we denote by $\norm{}_\alpha$ the unique $\heis$-invariant norm in the homothety class of $\supnorm{T_g(\cdot)}$ that is normalize at $\textbf{1}_{\Z_p^d}(x)$.
\end{defn} 

\subsection{The main results}
Our deepest results are Theorems \ref{thm_strong_minimality} and \ref{thm_strong_minimality_Z_p} below.
\begin{thm}[Rigidity]\label{thm_strong_minimality}
Let $\alpha\in\Grassmannian$.
If $\norm{}\in \invNorms$ is an $\heis$-invariant norm on $\schw$ that is dominated by $\norm{}_\alpha$, then $\norm{}=r\cdot \norm{}_\alpha$ for some constant $r>0$.    
\end{thm} 
In particular, each of the norms $\norm{}_\alpha$ is locally maximal at every non-zero vector in the completion of $\schw$ in it. 

The following proposition gives some basic properties of the completions of $\schw$ by a norm $\norm{}_\alpha$.
The first property follows from Theorem \ref{thm_strong_minimality} in conjunction with Theorem \ref{thm_strongly_cyclic_spaces} while the other are simpler.
\begin{prop}\label{topologically_irreducible}
Let $\alpha\in\Grassmannian$ and $\norm{}_\alpha$ the corresponding norm.
    \begin{enumerate}
    \item The completion $\completion{\schw}{\norm{}_\alpha}$ is a strongly irreducible Banach representation of $\heis$.
    \item The smooth part of $\completion{\schw}{\norm{}_\alpha}$ is precisely $\schw$.
    \item Let $\beta\in\Grassmannian$.
    The space of continuous $\heis$-equivariant maps from $\completion{\schw}{\norm{}_\alpha}$ to $\completion{\schw}{\norm{}_\beta}$ is given by
    \[Hom_{\heis}\sog{\completion{\schw}{\norm{}_\alpha},\completion{\schw}{\norm{}_\beta}}\simeq
    \begin{cases}
    \C_p & \alpha=\beta\\
    0    & \alpha\neq \beta
    \end{cases}.\]
    \end{enumerate}
\end{prop} 
Theorem \ref{thm_strong_minimality} and Proposition \ref{topologically_irreducible} form a $p$-adic analog of a classical result about unitary representations:
\begin{thm}[Classical known result]
Let $\schw^{\C}$ denote the space of $\C$-valued Schwartz functions on $\Q_p^d$, and $\rho_\psi^{\C}$ the complex Schrödinger representation.
Then, up to a positive scalar, there exists a unique $\heis$-invariant unitary structure on $\schw^{\C}$.
The completion with respect to the associated norm is topologically irreducible and its smooth part is the space $\schw^{\C}$.
\end{thm} 

Using Theorem \ref{Thm_strong_irreducibility} we will derive the following rigidity property.

\begin{thm}\label{thm_Rigidity}
Let $\alpha\in\Grassmannian$.
Let $(B,\norm{})$ be a topologically irreducible Banach representation of $\heis$ (see Definition \ref{def_Banach_rep}).
Assume that we are in one of the two following cases.
    \begin{enumerate}
    \item $F:B\map \completion{\schw}{\norm{}_\alpha}$ is a non-zero continuous map of representations.
    \item $F:\completion{\schw}{\norm{}_\alpha}\map B$ is a non-zero continuous map of representations.
    \end{enumerate}
Then $F$ is an isomorphism.
Moreover, there exists $r>0$ such that by replacing $\norm{}$ with $r\cdot \norm{}$, $F$ becomes an isometric isomorphism.
\end{thm} 

In order to prove Theorem \ref{thm_strong_minimality} we will prove its $\Z_p$-analog.
Let $\schw(\Z_p^d)$ denote the space of locally constant, $\C_p$-valued functions on $\Z_p^d$.
The sup norm on $\schw(\Z_p^d)$ is invariant under translations and under multiplication by the smooth characters of $\Z_p^d$.
Here, as before, a smooth character of $\Z_p^d$ is a homomorphism $\chi:(\Z_p^d,+)\map \C_p^\times$ with an open kernel.

\begin{thm}\label{thm_strong_minimality_Z_p}
Let $\norm{}$ be a norm on $\schw(\Z_p^d)$ that is dominated by the sup norm and invariant under translations and multiplication by smooth characters.
Then $\norm{}=r\cdot \supnorm{}$ for some $r>0$.
\end{thm} 

\subsection{A new proof of the main results of Fresnel and de Mathan}\label{subsection_new_proof_Fresnel_de_Mathan}
Our method gives a new proof of the main results in \cite{demathan}.
In that paper, Fresnel and de Mathan studied the Fourier transform 
\[\mathcal{F}:C_0(\Q_p/\Z_p)\map C(\Z_p)\]
attached to a smooth character $\psi:(\Q_p,+)\map \C_p^\times$ with $\ker(\psi)=\Z_p$.
Here, $C_0(\Q_p/\Z_p)$ denotes the space of $\C_p$-valued functions on $\Q_p/\Z_p$ which go to zero at infinity, and is equipped with the sup norm.
The main results in \cite{demathan} are stated in the following theorem.
\begin{thm}
The Fourier transform $\mathcal{F}$ is surjective and is not injective.
Moreover, if $K$ denotes its kernel, the induced map 
\[C_0(\Q_p/\Z_p)/K\map C(\Z_p)\]
is a surjective isometry.
\end{thm} 

    \begin{proof}
    Let $\mathcal{H}(\Z_p)$ be the following subgroup of the Heisenberg group $\mathcal{H}_3(\Q_p)$,
    \[\heis(\Z_p)=\braces{[a,b,t]\ |\ a\in\Z_p}.\]
    Then $\heis(\Z_p)$ acts on $C(\Z_p)$ by the usual rule
    \[([a,b,t]f)(x)=\psi(t+bx)\cdot f(x+a),\]
    namely, by translations and multiplication by smooth characters.
    The group $\heis(\Z_p)$ also acts on $C(\Q_p/\Z_p)_0$ by the rule
    \[([a,b,t]g)(x)=\psi(t-ab+ax)\cdot g(x-b).\]
    It is easy to verify that \[\mathcal{F}:C_0(\Q_p/\Z_p)\map C(\Z_p)\]
    is a continuous homomorphism of Banach representations of $\heis(\Z_p)$.
    By Theorem \ref{thm_strong_minimality_Z_p}, the sup norm on $C(\Z_p)$ is locally maximal with respect to any $f\in C(\Z_p)$ with $\supnorm{f}=1$.
    It follows from Theorem \ref{thm_strongly_cyclic_spaces} that any non-zero $f\in C(\Z_p)$ is strongly cyclic, hence that $C(\Z_p)$ is a strongly irreducible representation.
    By theorem \ref{Thm_strong_irreducibility}, $\mathcal{F}$ is surjective.
    If $\mathcal{F}$ were also injective, it would be, by the open mapping theorem, an isomorphism of Banach spaces.
    To see that this is not true, consider the characteristic functions $\phi_n(x):=\textbf{1}_{p^{-n}\Z_p}(x)\in C_0(\Q_p/\Z_p)$.
    Then $\supnorm{\phi_n}=1$, while $\supnorm{\mathcal{F}(\phi_n)}=\supnorm{p^n\cdot\textbf{1}_{p^n\Z_p}}=p^{-n}$.
    Therefore, $\mathcal{F}$ is not injective.
    Finally, denoting the kernel of $\mathcal{F}$ by $K$, we have the induced isomorphism of Banach representations
    \[C(\Q_p/\Z_p)_0/K\map C(\Z_p).\]
    By \ref{thm_strong_minimality_Z_p}, there exists a real number $r>0$ such that by taking the norm $r\cdot \supnorm{}$ on $C(\Z_p)$, the above isomorphism is an isometry.
    To show that $r=1$, it is enough to show that the image of the characteristic function $\phi_0$ in the quotient $C_0(\Q_p/\Z_p)/K$ has norm $1$.
    Note that $\phi_0$ is a strongly cyclic vector in $C_0(\Q_p/\Z_p)$, and that $\supnorm{}$ is a normalized and locally maximal at $\phi_0$.
    Thus, by Proposition \ref{prop_open_unit_ball_of_strongly_cyclic}, the open unit ball around $\phi_0$ in $C_0(\Q_p/\Z_p)$ consists of strongly cyclic vectors.
    In particular, all elements of $K$ are at distance at least one from $\phi_0$.
    It follows that the image of $\phi_0$ in the quotient has norm $1$.
    Therefore, $r=1$.
    \end{proof}
    
    \begin{remark}
    In \cite{demathan}, Fresnel and de Mathan first show that $\mathcal{F}$ is not injective by constructing non-zero elements in the kernel of $\mathcal{F}$.
    These elements have some special properties which then enable them to show that $\mathcal{F}$ is surjective.
    \end{remark}

\section{Proofs of the main results}
Theorems \ref{thm_strong_minimality_Z_p} will be proved in section $6$.
In this section we explain how to derive Theorem \ref{thm_Rigidity} and Proposition \ref{topologically_irreducible} from it, and perform easy reduction steps towards the proof in section $6$.

\subsection{Proof of Proposition \ref{topologically_irreducible}}

    \begin{proof}
    Let $\alpha\in\Grassmannian$.
    $(1)$. By Theorem \ref{thm_strong_minimality}, the norm $\norm{}_\alpha$ is locally maximal at $f$, for any $f\in \completion{\schw}{\norm{}_\alpha}$ with $\norm{f}_\alpha=1$.
    By Theorem \ref{Thm_strong_irreducibility} it is enough to show that $\completion{\schw}{\norm{}_\alpha}$ is topologically irreducible.
    Let $W$ be a proper closed sub-representation of $\completion{\schw}{\norm{}_\alpha}$.
    The quotient norm  on $\completion{\schw}{\norm{}_\alpha}/W$ induces an invariant semi-norm $\norm{}$ on $\schw$ that is dominated by $\norm{}_\alpha$.
    Since $\schw$ is irreducible, $\norm{}$ is a norm and by Theorem \ref{thm_strong_minimality}, $\norm{}=r\cdot \norm{}_\alpha$ for some $r>0$.
    Thus, $\norm{}$ is actually a norm and $W=0$.
    
    $(2)$. The claim is clear for $\completion{\schw}{\supnorm{}}=C_0(\Q_p^d)$.
    We will show that the general case follows from this one.
    Let $g\in \mathrm{Sp}_{2d}(\Q_p)$ such that $\alpha=Pg$, and let $T_g$ be a corresponding intertwining operator, normalized such that $\norm{}_\alpha=\supnorm{T_g(\cdot )}$ on $\schw$.
    Then $T_g$ extends to an isometric isomorphism $T_g:\completion{\schw}{\norm{}_\alpha}\map C_0(\Q_p^d)$.
    Although $T_g$ is not $\heis$-equivariant, it satisfies 
    \[T_g([w,t]f)=[wg,t]T_g(f).\]
    In particular, $T_g(f)$ is a smooth vector in $C_0(\Q_p^d)$ if and only if $f$ is a smooth vector in $\completion{\schw}{\norm{}_\alpha}$.
    
    $(3).$ Let $T:\completion{\schw}{\norm{}_\alpha}\map\completion{\schw}{\norm{}_\beta}$ be a continuous $\heis$-equivariant map.
    By the previous part, the restriction of $T$ to $\schw\subset \completion{\schw}{\norm{}_\alpha}$ is an $\heis$-equivariant map $T':\schw\map \schw$.
    By Schur's lemma for smooth representations, $T'$ is multiplication by a constant.
    If this constant is non-zero, it means that $\norm{}_\alpha$ and $\norm{}_\beta$, considered on $\schw$, are homothetic.
    By Proposition \ref{prop_norms_Grassmannian} this could only be the case if $\alpha=\beta$.
    Thus, if $\alpha\neq \beta$, $T=0$.
    If $\alpha=\beta$ then by continuity, $T$ is a multiplication by a scalar.
    \end{proof}

\subsection{Proof of Theorem \ref{thm_Rigidity}}

    \begin{proof}
    Let $(B,\norm{})$ be an irreducible Banach representation of $\heis$.
    Assume we are in the first case, and let $F:B\map \completion{\schw}{\norm{}_\alpha}$ be a continuous map of representations.
    Assume that $F$ is non-zero.
    Since $B$ is topologically irreducible, the kernel of $F$ is zero, so $F$ is injective. 
    By Proposition \ref{topologically_irreducible} and Theorem \ref{Thm_strong_irreducibility}, $F$ is surjective.
    Thus, $F$ is an isomorphism.
    The norm $\norm{F^{-1}(\cdot)}$ is an $\heis$-invariant norm on $\completion{\schw}{\norm{}_\alpha}$ that is dominated by $\norm{}_\alpha$.
    By theorem \ref{thm_strong_minimality}, there exists $r>0$ such that $r\cdot \norm{F^{-1}(\cdot)}=\norm{}_\alpha$.
    Replacing $\norm{}$ by $r\cdot\norm{}$, $F$ becomes an isometry.    
        
    Assume we are in the second case and let $F:\completion{\schw}{\norm{}_\alpha}\map B$ a continuous map of representations.
    Assume that $F$ is non-zero.
    The norm $\norm{F(\cdot)}$ is an $\heis$-invariant norm on $\completion{\schw}{\norm{}_\alpha}$ that is dominated by $\norm{}_\alpha$.
    By Theorem \ref{thm_strong_minimality}, there exists $r>0$ such that $r\cdot\norm{F(\cdot)}=\norm{}_\alpha$.
    Replacing $\norm{}$ by $r\cdot \norm{}$, $F$ becomes an isometry.
    The image of $F$ is therefore a closed sub-representation of $B$, and since $B$ is topologically irreducible, $F$ is surjective.
    \end{proof} 

\subsection{Reduction steps}
The goal of this section is to show that Theorem \ref{thm_strong_minimality} and Theorem \ref{thm_strong_minimality_Z_p} follow from the particular case of Theorem \ref{thm_strong_minimality_Z_p} with $d=1$.

\begin{prop}
If Theorem \ref{thm_strong_minimality} holds for the sup norm then it holds for $\norm{}_\alpha$ for any $\alpha\in\Grassmannian$.
\end{prop} 
    \begin{proof}
    Let $Pg=\alpha\in\Grassmannian$, where $g\in \mathrm{Sp}_{2d}(\Q_p)$, and let $T_g$ be an intertwining operator such that $\norm{}_\alpha=\supnorm{T_g(\cdot)}$.
    Let $\norm{}$ be an $\heis$-invariant norm on $\schw$, dominated by $\norm{}_\alpha$.
    
    The operators $T_g$ and $(T_g)^{-1}$ act on $\invNorms$ and preserve order.
    In particular, $\norm{(T_g)^{-1}(\cdot)}$ is an $\heis$-invariant norm, dominated by the sup norm.
    By assumption, $\norm{(T_g)^{-1}(\cdot)}=r\cdot \supnorm{}$ for some $r>0$.
    Thus, $\norm{}=r\cdot \norm{}_\alpha$.    
    \end{proof} 

Next we show that Theorem \ref{thm_strong_minimality} for the sup norm follows from Theorem \ref{thm_strong_minimality_Z_p}.
\begin{prop}
Assume that Theorem \ref{thm_strong_minimality_Z_p} holds.
Let $\norm{}$ be an $\heis$-invariant norm on $\schw$ that is dominated by the sup norm.
Then $\norm{}=r\cdot \supnorm{}$ for some $r>0$.
\end{prop} 
    \begin{proof}
    For any $n\in\N$ we denote $V_n=\schw(p^{-n}\Z_p^d)$ and think about $V_n$ as the subspace of $\schw(\Q_p^d)$ of functions supported on the disc $p^{-n}\Z_p^d$.
    The restriction of $\norm{}$ to $V_n$ is invariant under translations by $p^{-n}\Z_p^d$ and multiplication by smooth characters.
    By Theorem \ref{thm_strong_minimality_Z_p} and an obvious change of variables, there exists $r_n>0$ such that $\norm{f}=r_n\cdot\supnorm{f}$ for any $f\in V_n$.
    The function $\textbf{1}_{\Z_p^d}(x)$ lies in any of the $V_n$, so the numbers $(r_n)_{n\in\N}$ must be equal to the same $r$.
    Then $\norm{f}=r\cdot \supnorm{f}$ for any compactly supported function $f$.
    \end{proof}   
    
\begin{prop}
If Theorem \ref{thm_strong_minimality_Z_p} holds for $\Z_p$ then it holds for $\Z_p^d$ for any $d$.
\end{prop} 
    \begin{proof}
    The proof is by induction, the case $d=1$ being assumed to be true.
    Let $d>1$ and assume that Theorem \ref{thm_strong_minimality_Z_p} holds for $d-1$.
    Let $\norm{}$ be a norm on $\schw(\Z_p^d)$ that is invariant under translations and multiplication by smooth characters, dominated by the sup norm and normalized on $\textbf{1}_{\Z_p^d}(x)$.
    By Proposition \ref{prop_weak_minimality} it is enough to show that $\norm{}\leq \supnorm{}$.
    The latter follows if we show that for any $n$, $\norm{\textbf{1}_{p^n\Z_p^d}(x)}=1$, where $\textbf{1}_{p^n\Z_p^d}(x)$ is the characteristic function $p^n\Z_p^d$.
    
    Let $0<n\in \N$.
    Let $P_d$ by the projection $\alpha:\Z_p^d\map \Z_p$ given by $\alpha(a_1,...,a_d)=a_d$, and denote by $P_d^*$ the induced map $P_d^*:\schw(\Z_p)\map \schw(\Z_p^d)$.
    It is easy to see that the norm $\norm{P_d^*(\cdot)}$ on $\schw(\Z_p)$ is invariant under translations and multiplication by smooth characters, dominated by the sup norm and normalized at $\textbf{1}_{\Z_p}(x)$.
    Thus, $\norm{P_d^*(f)}=\supnorm{f}$ for any $f\in \schw(\Z_p)$.
    In particular,
    \[\norm{P_d^*(\textbf{1}_{p^n\cdot\Z_p}(x))}=1,\]
    Note that
    \[P_d^*(\textbf{1}_{p^n\cdot\Z_p}(x))=\textbf{1}_{\Z_p^{d-1}\times (p^n\cdot\Z_p)}(x).\]
    
    Now, consider the Projection $\beta:\Z_p^{d-1}\times (p^n\cdot \Z_p)\map\Z_p^{d-1}$ given by $\beta(a_1,...,a_{d-1},a_d)=(a_1,..,a_{d-1})$, and the induced map $\beta^*:C(\Z_p^{d-1})\map C(\Z_p^d)$.
    By the previous lemma, the norm $\norm{\beta^*(\cdot)}$ on $\schw(\Z_p^{d-1})$ is invariant under translations and multiplication by smooth character and dominated by the sup norm.
    Since 
    \[\beta^*(\textbf{1}_{\Z_p^{d-1}}(x))=\textbf{1}_{\Z_p^{d-1}\times(p^n\cdot \Z_p)}(x)\]
    and since $\norm{\textbf{1}_{\Z_p^{d-1}\times(p^n\cdot \Z_p)}(x)}=1$, we deduce that $\norm{\beta^*(\cdot)}$ is normalized on $\textbf{1}_{\Z_p^{d-1}}$.
    Therefore $\norm{\beta^*(\cdot)}=\supnorm{}$.
    In particular,
    \[\norm{\beta^*(\textbf{1}_{p^n\cdot \Z_p^{d-1}}(x))}=1.\]
    Note that 
    \[\beta^*(\textbf{1}_{p^n\cdot \Z_p^{d-1}}(x))=\textbf{1}_{p^n\cdot \Z_p^d}(x).\]
    Thus, we proved that $\norm{\textbf{1}_{p^n\cdot \Z_p^d}(x)}=1$.
    \end{proof} 

It remains to prove Theorem \ref{thm_strong_minimality_Z_p} for $\Z_p$.
This is done in the next section.

\section{Proof of Theorem \ref{thm_strong_minimality_Z_p} for $\mathbb{Z}_p$}

In this section we prove Theorem \ref{thm_strong_minimality_Z_p} for $\Z_p$.
We begin by noting that in the formulation of Theorem \ref{thm_strong_minimality_Z_p}, the space $\schw(\Z_p)$ can be replaced by $C(\Z_p)$, which is its completion in the sup norm.
In this section we will work with $C(\Z_p)$ since this allows us to use functions, such as polynomials, which are not in $\schw(\Z_p)$.

Clearly, Theorem \ref{thm_strong_minimality_Z_p} follows if we know that $\supnorm{}$ is both weakly minimal and locally maximal at $\textbf{1}_{\Z_p}(x)$.
That the sup norm is weakly minimal at $\textbf{1}_{\Z_p}(x)$ follows from Proposition \ref{prop_weak_minimality}.
Thus, it remains to show local maximality.
        
The proof uses two main ingredients:
    \begin{enumerate}
    \item The \textbf{growth modulus} of a norm.
    This is a real valued function associated with norms on $C(\Z_p)$ that are dominated by the sup norm.
    \item The \textbf{$q$-Mahler bases}.
    To each $q\in \C_p$ with $\pabs{q-1}<1$, there corresponds a basis of $C(\Z_p)$ called the $q$-Mahler basis which shares some nice properties with the Mahler basis: $\braces*{\binom{x}{n}\ |\ n\geq 0}$.
    The $q$-Mahler bases can be viewed as a family of deformations of the Mahler basis.
    \end{enumerate}

\subsection{The growth modulus of a norm}
The beginning of this section is an adaptation of \cite{robert2013course}, chapter $6$, part $1.4$.

Let $(a_n)_{n=0}^\infty$ be a bounded sequence of non-negative real numbers.
The growth modulus associated with the sequence $(a_n)_{n=0}^\infty$ is the function
\[r\mapsto\sup_na_nr^n\]
defined on the interval $[0,1]$.
It is a continuous, non-decreasing and convex function (part of the \textbf{Classical Lemma} on \cite{robert2013course}, p.292).

We say that a real number $0<r<1$ is regular with respect to the sequence $(a_n)_{n=0}^\infty$ if there exists $n$ such that $a_nr^n>a_mr^m$ for any $m\neq n$.
Otherwise, we call $r$ a critical value (with respect to the sequence).

We remark that if $r$ is a regular value and $a_nr^n>a_mr^m$ for any $m\neq n$, there exists some interval containing $r$ on which the growth modulus is equal to $a_nx^n$.
In particular the growth modulus is smooth at the regular values.

The fundamental lemma about critical values is the following.
\begin{lem}\label{discreteness_of_critical_values}
Assume that $(a_n)_{n=0}^\infty$ is not the zero sequence.
The set of critical values is discrete in $[0,1)$. 
\end{lem} 
    \begin{proof}
    Let $0<r<1$.
    We will show that there are only finitely many critical values smaller than $r$.
    Let $n$ be such that $a_nr^n\geq a_mr^n$ for any $m$.
    Note that $a_n\neq 0$.
    Let $0<s<r$.
    Then for any $N>n$
    \[a_nr^n\geq a_Nr^N \gorer \frac{a_N}{a_n}\leq r^{n-N}<s^{n-N} \gorer a_ns^n>a_Ns^N.\]
    Thus, if $s<r$ is a critical value and $i,j\in\Z_{\geq0}$ are such that $a_is^i=a_js^j\geq a_ks^k$ for any $k$, then $i,j\leq n$.
    In this case $a_j\neq 0$ and
    \[s=\sog{\frac{a_i}{a_j}}^{\frac{1}{j-i}},\ \ \ i,j\leq n.\]
    There are only finitely many such values.   
    \end{proof} 

Until the end of this subsection, fix a norm $\norm{}$ on $C(\Z_p)$ that is dominated by the sup norm and normalized at $\textbf{1}_{\Z_p}(x)$.
Let $\binom{x}{n}$ be the $n$-th binomial polynomial.
Under the assumptions on $\norm{}$, the sequence $\sog{\norm{\binom{x}{n}}}_{n\geq 0}$ is bounded. 
We define the growth modulus of the norm $\norm{}$ to be the growth modulus of that sequence.
We denote the growth modulus of $\norm{}$ by $G_{\norm{}}(r)$.
Explicitly, $G_{\norm{}}(r):[0,1]\map\R$ is the function 
\[G_{\norm{}}(r)=\sup_{n\geq 0}\sog{\norm{\binom{x}{n}}\cdot r^n}.\]
We call $r\in[0,1]$ a regular (resp. critical) value for the norm $\norm{}$ if it is regular (resp. critical) with respect to the sequence $\sog{\norm{\binom{x}{n}}}_{n\geq 0}$.

The connection between the growth modulus of $\norm{}$ and the study of the norm itself comes from the work of Mahler.
We recall the basic facts about the Mahler basis.

\begin{thm}[Mahler, \cite{Mahler}]
Any $f\in C(\Z_p)$ can be written as 
\[f(x)=\sum_{n=0}^\infty a_n\cdot \binom{x}{n}\]
where $\limit{n}a_n=0$ and the sum converges to $f$ in the sup norm.
Moreover, $\supnorm{f}=\max_{n}\pabs{a_n}$.
\end{thm} 

The following proposition immediately follows.
\begin{prop}
Let $M$ be the smallest number such that $\norm{}\leq M\cdot \supnorm{}$.
Then $M=G_{\norm{}}(1)$.
\end{prop} 
We conclude this subsection with the following proposition.
\begin{prop}\label{powers_and_growth_modulus}
Let $\norm{}$ be a norm on $C(\Z_p)$ dominated by the sup norm.
Let $q\in \C_p $ with $r:=\abs{q-1}_p<1$.
Then $\norm{q^x}\leq G_{\norm{}}(r)$.
Moreover, if $r$ is a regular value for the norm $\norm{}$ then
\[\norm{q^x}=G_{\norm{}}(r).\]
\end{prop} 
    \begin{proof}
    This is a simple consequence of the non-archimedean triangle inequality.
    Using the Mahler expansion of the function $q^x$
        \begin{equation}\label{equation_growth_modulus_proof}
        \norm{q^x}=\norm{\sum_{k=0}^\infty(q-1)^k\binom{x}{k}}\leq \sup_{k \geq 0} \sog{\abs{q-1}_p^k\norm{\binom{x}{k}}}=G_{\norm{}}(r).  
        \end{equation} 
    If $r=\abs{q-1}_p$ is a regular value, there exists $n\geq 0$ such that 
    \[r^n\norm{\binom{x}{n}}>r^m\norm{\binom{x}{m}}\]
    for any $m\neq n$, and therefore we have an equality instead of inequality in \ref{equation_growth_modulus_proof}.
    \end{proof}
    \begin{exmp}
    $\ $
        \begin{enumerate}
        \item The growth modulus of the sup norm is constant $G_{\norm{}}(r)=1$.
        \item Assume that $\norm{}$ is invariant by multiplication by smooth characters and that $G_{\norm{}}(1)>1$.
        Then, for any $N$ large enough and $\zeta$ a root of unity of order $p^N$, $r=\pabs{\zeta-1}$ is a critical value for the norm $\norm{}$.
        Indeed, for $N$ large enough, $G_{\norm{}}(r)>1$ while $\norm{\zeta^x}=1$, so $r$ is a critical value by the previous proposition.
        \end{enumerate}
    \end{exmp} 

    \begin{remark}
    In general, $\norm{q^x}$ is not a function of $\abs{1-q}_p$, i.e. it might be the case that $\norm{q_1^x}\neq\norm{q_2^x}$ while $\abs{q_1-1}_p=\abs{q_2-1}_p$.
    \end{remark} 

\subsection{$q$-Mahler bases}
We briefly recall the $q$-analog terminology, the $q$-Mahler bases and the expansion formula for exponents in these bases.
This subsection is self contained.
For a more thorough exposition to the $q$-analog formalism and its properties we refer to \cite{q_calculus,special_functions}.
For more on the $q$-analog of the Mahler basis see \cite{conrad2000q}.

Let $q$ be an indeterminate.
The $q$-analog of the natural number $n$ is the following expression in $\Z[q]$
\[[n]_q=\frac{1-q^n}{1-q}=1+q+...+q^{n-1}.\]
The $q$-analog of the factorial of $n$ is 
\[[n]_q!=[1]_q\cdot [2]_q\cdot ....\cdot [n]_q\]
and the $q$-binomial coefficients, also known as Gaussian binomial coefficients, are defined by the analogous formula
\[\qbinom{n}{k}{q}=\frac{[n]_q!}{[k]_q!\cdot [n-k]_q!}\]
whenever $0\leq k\leq n$, and zero otherwise.
The $q$-Pochhammer symbol is the expression
\[(a;q)_n=\prod_{i=0}^{n-1}(1-aq^i).\]
When $a=q$ we get
\[(q;q)_n=\prod_{i=1}^{n}(1-q^i).\]
By expanding the terms in the definition, it is easy to verify that
\[\qbinom{n}{k}{q}=\frac{(q;q)_n}{(q;q)_{k}(q;q)_{n-k}}.\]

The $q$-Pascal identity
\begin{equation}\label{q_Pascal}
\qbinom{n+1}{k+1}{q}=\qbinom{n}{k+1}{q}+q^{n-k}\cdot\qbinom{n}{k}{q},
\end{equation}
implies, by induction, that $\qbinom{n}{k}{q}$ is a polynomial in $q$ with integer coefficients.

From now on $q$ will not be an indeterminate but an element in $\C_p$ such that $\abs{q-1}_p<1$.
The map $n\mapsto \qbinom{n}{k}{q}$ is continuous with respect to the $p$-adic topologies on $\Z$ and on $\C_p$, and therefore extends to a map $x\mapsto \qbinom{x}{k}{q}$ that lies in $C(\Z_p)$.

Since for any $x\in \N$ the expression $\qbinom{x}{k}{q}$ is a polynomial with integral coefficients in $q$, we have $\pabs{\qbinom{x}{k}{q}}\leq 1$.
By continuity, $\supnorm{\qbinom{x}{k}{q}}\leq 1$.
Substituting $x=k$ we see that 
\[\supnorm{\qbinom{x}{k}{q}}=1.\]

Note that if  $q$ is not a root of unity the term $(q;q)_k$ is non-zero for any $k$, so
\[\qbinom{x}{k}{q}=\frac{(1-q^{x-(k-1)})\cdot(1-q^{x-(k-2)})\cdot...\cdot(1-q^x)}{(1-q)(1-q^2)...(1-q^k)}.\]

We will need two results about $q$-binomial functions.
The first is the $q$-analog of Mahler's theorem.
The second is the expansion of an exponent $\zeta^x$ with respect to the $q$-Mahler basis.
Both results appear in \cite{conrad2000q}, the first is a combination of Theorem 3.3 and Theorem 4.1, and the second is the example at the beginning of page 14.
For completeness we will prove both results.
\begin{thm}\label{q_Mahler}
Let $q\in\C_p$ with $\abs{q-1}_p<1$.
Then for any function $f\in C(\Z_p)$ there exists a unique sequence $(a_n)_{n=0}^\infty$ of numbers in $\C_p$ such that the series
\[\sum_{k=0}^\infty a_k\qbinom{x}{k}{q}\]
converges in the sup norm to $f$ (in particular $\limit{k}a_k=0$).
Moreover,
\[\supnorm{f}=\max_{k\geq 0}\pabs{a_k}.\]
\end{thm} 
    \begin{proof}
    Consider the operator $T=\frac{\Delta}{q^x}$ on $C(\Z_p)$, where $\Delta$ is the forward difference operator.
    Thus,
    \[Tf(x)=\frac{f(x+1)-f(x)}{q^x}.\]
    We begin by showing that for any $f\in\C(\Z_p)$, the sequence $\sog{T^nf(0)}_{n=0}^\infty$ converges to zero.
    Afterwards we will construct the sequence $(a_n)_{n=0}^\infty$ from $\sog{T^nf(0)}_{n=0}^\infty$.
    
    Recall that $\pabs{q-1}<1$, and denote $r=\pabs{q-1}$.
    Denote $r=\abs{q-1}_p$ and recall the assumption that $r<1$.
    We consider the quotient space
    \[W=\quot{\braces{f\in C(\Z_p)\ |\ \supnorm{f}\leq 1}}{\braces{f\in C(\Z_p)\ |\ \supnorm{f}\leq r}}.\]
    Its elements can be realized as locally constant functions on $\Z_p$ with values in $\mathcal{O}_{\C_p}/(q-1)\mathcal{O}_{\C_p}$.
    Since the operator $T$ is norm reducing, i.e. $\supnorm{T(f)}\leq \supnorm{f}$ for any $f\in C(\Z_p)$, $T$ induces an operator on $W$.
    Since the image of $q^x$ in $W$ is the constant function $1$, the operator $T$ reduces in $W$ to the forward difference operator $\Delta$.
    If $v\in W$, there exists some number $N$ such that $\Delta^{p^N}v= 0$.
    Thus, for any functions $f\in C(\Z_p)$ there exists $N>0$ such that 
    \[\supnorm{T^{p^N}f}\leq r\cdot \supnorm{f}.\]
    Together with the fact that $T$ is norm reducing, it follows that $\limit{n}\supnorm{T^nf}=0$.
    In particular, $\limit{n}T^nf(0)=0$.    
    
    By rearranging the $q$-Pascal identity \ref{q_Pascal} we get
    
    \[\frac{q^{\binom{k}{2}}\qbinom{n+1}{k}{q}-q^{\binom{k}{2}}\qbinom{n}{k}{q}}{q^n}=q^{\binom{k-1}{2}}\qbinom{n}{k-1}{q}.\]
    Continuity with respect to $n$ implies that \[T\sog{q^{\binom{k}{2}}\qbinom{x}{k}{q}}=q^{\binom{k-1}{2}}\qbinom{x}{k-1}{q}\]
    for any $k\geq 1$.
    When $k=1$ the function $q^0\qbinom{x}{0}{q}$ is just the constant function $1$, and clearly $T(1)=0$.
    
    Let $f\in C(\Z_p)$ and denote $a_n=q^{\binom{n}{2}}(T^nf)(0)$.
    The series 
    \[h(x)=\sum_{k=0}^\infty a_k\qbinom{x}{k}{q}\]
     converges in $C(\Z_p)$, and we have
     \[T^nh(0)=T^nf(0),\]
     for any $n\geq 0$.
    Since $h(0)=f(0)$, it follows that $h(n)=f(n)$ for any $n\geq 0$.
    By continuity we must have $h=f$.
    Thus,
    \[f(x)=\sum_{k=0}^\infty (T^kf)(0)q^{\binom{k}{2}}\qbinom{x}{k}{q}.\]
    This formula implies that $\supnorm{f}\leq\max_{k}\abs{a_k}_p$.
    The inequality in the other direction follows from the fact that $T$ is norm-reducing, so $\abs{a_k}_p=\abs{(T^nf)(0)}_p\leq \supnorm{T^nf}\leq \supnorm{f}$.
    \end{proof} 
\begin{defn}
We denote 
\[\pcoeff{\zeta}{q}_k=(\zeta-1)(\zeta-q^1)...(\zeta-q^{k-1})=(-1)^k\cdot q^{\binom{k}{2}}\cdot (\zeta;q^{-1})_k,\]
for $k>0$ and $\pcoeff{\zeta}{q}_0=1$.

\end{defn} 
\begin{cor}\label{q_expansion_of_powers}
Let $\zeta,q\in\C_p$ with $\abs{q-1}_p<1$ and $\abs{\zeta-1}_p<1$.
Then
\[\zeta^x=\sum_{k=0}^\infty \pcoeff{\zeta}{q}_k\cdot\qbinom{x}{k}{q}.\]
\end{cor} 
    \begin{proof}
    We have $T^0(\zeta^x)(0)=\zeta^0=1=[\zeta,q]_0$.
    Compute 
    \[T(\zeta^x)=\frac{\zeta^{x+1}-\zeta^x}{q^x}=(\zeta-1)\sog{\frac{\zeta}{q}}^x.\]
    By induction:
    \[T^k(\zeta^x)=(\zeta-1)(\frac{\zeta}{q}-1)...(\frac{\zeta}{q^{k-1}}-1)\sog{\frac{\zeta}{q^k}}^x.\]
    By the proof of Theorem \ref{q_Mahler}, the coefficient of $\qbinom{x}{k}{q}$ in the expansion of $\zeta^x$ is 
    \[q^{\binom{k}{2}}\cdot (T^kf)(0)=q^{\binom{k}{2}}(\zeta-1)(\frac{\zeta}{q}-1)...(\frac{\zeta}{q^{k-1}}-1)\sog{\frac{\zeta}{q^k}}^0=(\zeta-1)(\zeta-q)...(\zeta-q^{k-1}).\]
    \end{proof}

\subsection{The $p$-adic valuation of $(\zeta;\zeta)_n$ when $\zeta$ is a root of unity}
Fix $N\in \N$ and let $\zeta$ be a primitive $p^{N}-th$ root of unity in $\fld{C}_p$. 
In this subsection we study the $p$-adic valuation of the expression
\[(\zeta;\zeta)_n=(1-\zeta)(1-\zeta^2)...(1-\zeta^n)\]
for $1\leq n<p^N$.
It will be convenient to denote $\lambda=-\log_p(\pabs{\zeta-1})$, so $\lambda=\frac{1}{p^{N-1}(p-1)}$.

The main goal is to prove the following result, which will be used later.
\begin{prop}\label{q_analog_evaluations_cor}
For any $p^8\leq n<p^N$,
\[\log_p(\abs{(\zeta;\zeta)_n}_p)\leq -\frac{\lambda}{4}n\log_p(n).\]
\end{prop} 

\begin{defn}
For a positive integer $n$ we define $\beta_p(n)$ to be
\[
\beta_p(n)=
\sum_{k=0}^{\infty}p^k\sog{\floor*{\frac{n}{p^k}}-\floor*{\frac{n}{p^{k+1}}}}.
\]
\end{defn}
Note that for any $n$ this sum is finite.

\begin{prop}\label{beta_function}
For any $1\leq n<p^N$ 
\[\log_p(\abs{(\zeta;\zeta)_n}_p)=-\lambda\cdot \beta_p(n).\]
\end{prop}

\begin{proof}
For any $1\leq m=ap^k<N$, where $p\nmid a$, we have
\[\log_p\sog{\pabs{\zeta^m-1}}=\log_p\sog{\pabs{\zeta^{p^k}-1}}=-\frac{1}{p^{N-k-1}(p-1)}=-p^k\cdot \lambda.\]
There are $\floor*{\frac{n}{p^k}}-\floor*{\frac{n}{p^{k+1}}}$ numbers between $1$ and $n$ that are divisible by $p^k$ but not by $p^{k+1}$.
Thus,
\begin{align*}
    \log_p\sog{\absolute*{(\zeta;\zeta)_n}_p}
    &=\sum_{i=1}^n\log_p\sog{\absolute*{1-\zeta^i}_p}
    =\sum_{k=0}^\infty\sog{ \floor*{\frac{n}{p^k}}-\floor*{\frac{n}{p^{k+1}}} }\cdot \log_p\sog{\absolute*{1-\zeta^{p^k}}_p}\\
    &=\sum_{k=0}^\infty\sog{ \floor*{\frac{n}{p^k}}-\floor*{\frac{n}{p^{k+1}}} }\cdot \sog{-p^k\lambda}
    =-\lambda\cdot\beta_p(n).
    \end{align*} 
\end{proof}

\begin{prop}
For every $n$
\[
\beta_p(n)\geq n\log_p(n)\cdot \frac{p-1}{p}-\frac{np}{p-1}.
\]
\end{prop}

\begin{proof}
Denote $d=\floor{\log_p(n)}$.
Then
   \begin{align*}
    \beta_p(n)
    &=\sum_{k=0}^{d}p^k\sog{\floor*{\frac{n}{p^k}}-\floor*{\frac{n}{p^{k+1}}}}
    =\sum_{k=0}^{d-1}p^k\sog{\floor*{\frac{n}{p^k}}-\floor*{\frac{n}{p^{k+1}}}}+p^d\floor*{\frac{n}{p^d}}\\
    &\geq\sum_{k=0}^{d-1}p^k\sog{\frac{n}{p^k}-1-\frac{n}{p^{k+1}}}+p^d\sog{\frac{n} {p^d}-1}\\
    &=\sum_{k=0}^{d-1}\sog{n-p^k-\frac{n}{p}}+n-p^d
    =dn\sog{1-\frac{1}{p}}+n-\sum_{k=0}^{d}p^k\\
    &=dn\frac{p-1}{p}+n-\frac{p^{d+1}-1}{p-1}.
    \end{align*}
Since  $n\geq p^d$ and $d> \log_p(n)-1$,
    \begin{align*}
    dn\frac{p-1}{p}+n-\frac{p^{d+1}-1}{p-1}
    &\geq \sog{\log_p(n)-1} n\frac{p-1}{p}+n-\frac{pn-1}{p-1}\\
    &=n\log_p(n) \frac{p-1}{p}-n+\frac{n}{p}+n-\frac{pn-1}{p-1}\\
    &<  n\log_p(n)\frac{p-1}{p}-\frac{np}{p-1}
    \end{align*} 
\end{proof}

    \begin{proof}[Proof of Proposition \ref{q_analog_evaluations_cor}]
    Write
    \[\beta_p(n)\geq n\log_p(n)\cdot\frac{p-1}{p}-n\cdot\frac{p}{p-1}
    =\frac{1}{2}\sog{n\log_p(n)\cdot\frac{p-1}{p}}+\frac{1}{2}\sog{n\log_p(n)\cdot\frac{p-1}{p}}-n\cdot\frac{p}{p-1}.\]
    If $n\geq p^8$, then
    \[\frac{1}{2}\sog{n\log_p(n)\cdot\frac{p-1}{p}}-n\cdot\frac{p}{p-1}
    \geq\frac{1}{2} \sog{n\cdot 8\cdot\frac{p-1}{p}}-n\cdot\frac{p}{p-1}
    =n\cdot\frac{(p-1)(3p-1)}{p(p-1)}\geq 0.\]
    Therefore, 
    \[\beta_p(n)\geq  \frac{1}{2}\cdot\frac{p-1}{2}n\log_p(n)\geq \frac{1}{4}n\log_p(n).\]
    By Proposition \ref{beta_function},
    \[\log_{p}(\abs{(\zeta;\zeta)_n}_p)=-\lambda\cdot\beta_p(n)\leq -\frac{\lambda}{4}n\log_p(n).\]
    \end{proof} 

\subsection{Completing the proof of Theorem \ref{thm_strong_minimality_Z_p} for $\mathbb{Z}_p$}
Let $\norm{}$ be a norm on $C(\Z_p)$, dominated by the sup norm, normalized at $\textbf{1}_{\Z_p}(x)$ and invariant under multiplication by smooth characters of $\Z_p$.
Let $G_{\norm{}}(r)$ be the growth modulus of $\norm{}$.
We suppose that $G_{\norm{}}(1)>1$ and reach a contradiction.

By the assumption that $G_{\norm{}}(1)>1$, the continuity of $G_{\norm{}}(r)$ and the density of regular values (Proposition \ref{discreteness_of_critical_values}), there exists $h\in \C_p$ such that $s:=\abs{h}_p<1$ is a regular value for the norm $\norm{}$ and such that $G_{\norm{}}(s)>1$.
We may also assume that $s\geq p^{-1/(p-1)}$.
The last assumption can be written as $\log_p\sog{\frac{1}{s}}\leq \frac{1}{p-1}$.
We fix such $h$ and denote $s=\abs{h}_p$.
We remark that $h$ and $s$ depend only on the norm $\norm{}$.

From now on, $\zeta$ denotes a primitive $p^N$-th root of unity and $N$ is assumed to be very large (in a way that will be made explicit below).
We denote $\lambda=-\log_p(\pabs{\zeta-1})$ and $q=\zeta+h$.

Thus, $h$ is fixed and $\zeta$ is at our disposal, close as we wish to the circumference of the unit disc around $1$, and $q$ varies with $\zeta$ at a fixed distance $s$ from it.

The idea of the proof is to use the expansion 
\[\zeta^x=\sum_{k=0}^\infty\pcoeff{\zeta}{q}_k\cdot\qbinom{x}{k}{q}\]
to show, under the assumption that $N$ is very large, that
    \begin{equation}\label{main_inequality}
    \norm{\pcoeff{\zeta}{q}_1\cdot\qbinom{x}{1}{q}}
    >\norm{\pcoeff{\zeta}{q}_k\cdot\qbinom{x}{k}{q}}  
    \end{equation} 
for any $k\neq 1$.
Then, by the strong triangle inequality,
\[\norm{\zeta^x}
=\norm{\pcoeff{\zeta}{q}_1\cdot\qbinom{x}{1}{q}}
>\norm{\pcoeff{\zeta}{q}_0\cdot\qbinom{x}{0}{q}}
=\norm{\textbf{1}(x)}
=1\] 
which is a contradiction to the assumption that $\norm{}$ is invariant under multiplication by $\zeta^x$ and normalized at $\textbf{1}_{\Z_p}(x)$.

The proof of \ref{main_inequality} will be divided into three cases.
The first, $k=0$, is the easiest.
The second and third cases are when  $1<k<\frac{1}{\lambda}\log_p\sog{\frac{1}{\sqrt{s}}}$ and $k\geq \frac{1}{\lambda}\log_p\sog{\frac{1}{\sqrt{s}}}$ respectively.
In each of these cases we will need to use different type of inequalities.

\begin{prop}\label{prop_norm_of_q^x}
We have that $\norm{q^x}=G_{\norm{}}(s)$ and $\norm{q^{ax}}\leq \norm{q^x}$ for any $a\in\Z_p$.
\end{prop} 
    \begin{proof}
    Write
    \[q^x=(\zeta+h)^x=\zeta^x\sog{1+\frac{h}{\zeta}}^x.\]
    Since $\norm{}$ is invariant under multiplication by smooth characters
    \[\norm{\zeta^x\sog{1+\frac{h}{\zeta}}^x}=\norm{\sog{1+\frac{h}{\zeta}}^x}.\]
    Since $\abs{h/\zeta}_p=\abs{h}_p=s$ is a regular value for the norm $\norm{}$, we have, by Proposition \ref{powers_and_growth_modulus}, an equality
    \[\norm{\sog{1+\frac{h}{\zeta}}^x}=G_{\norm{}}(s).\]
    Thus, $\norm{q^x}=G_{\norm{}}(s)$.
    Let $a\in\Z_p$.
    To show that $\norm{q^{ax}}\leq \norm{q^x}$ we use the same trick.
    Write
    \[\norm{q^{ax}}=\norm{(\zeta+h)^{ax}}=\norm{\zeta^{ax}\cdot\sog{1+\frac{h}{\zeta}}^{ax}}=\norm{\sog{1+\frac{h}{\zeta}}^{ax}}=\norm{(1+h')^x},\]
    where $h'=(1+h/\zeta)^a-1$.
    Then $\abs{h'}_p\leq \abs{h}_p$.
    Since $G_{\norm{}}(r)$ is monotone increasing, and by proposition \ref{powers_and_growth_modulus}, 
    \[\norm{q^{ax}}=\norm{(1+h')^x}\leq G(\abs{h'}_p)\leq G(\abs{h}_p)=\norm{q^x}.\]
    \end{proof} 

\begin{prop}\label{some_p_adic_valuations}
Assume that $\pabs{1-\zeta}>s$.
    \begin{enumerate}
    \item Let $ s< r<1$.
    Then for any $1\leq i\leq\frac{1}{\lambda}\log_p\sog{\frac{1}{r}}$ 
    \[\abs{1-q^i}_p=\abs{1-\zeta^i}_p\geq r.\]
    \item For any $1<k\leq\frac{1}{\lambda}\log_p\sog{\frac{1}{\sqrt{s}}}$ 
    \[\abs{\pcoeff{\zeta}{q}_k}_p\leq \sqrt{s}\cdot\abs{(q;q)_k}_p\]
    \end{enumerate}
\end{prop} 
    \begin{proof}
    First, recall our assumption that $\log_p\sog{\frac{1}{s}}\leq\frac{1}{p-1}$.
    Then 
    \[\frac{1}{\lambda}\log_p\sog{\frac{1}{r}}\leq \frac{1}{\lambda}\log_p\sog{\frac{1}{s}}\leq \frac{1}{(p-1)\lambda}=p^{N-1},\]
    for any $s<r<1$.
    In particular, any indices $i$ and $k$ that appear in this proof are in $\braces{0,1,...,p^N-1}$, so the expressions $\zeta^i, \zeta^k$ are not equal to $1$.
    Second, note that if $N$ is not large enough, the interval $[1,\frac{1}{\lambda}\log_p\sog{\frac{1}{r}}]$ may be empty.
        
    (1). For any $i\geq 1$, $\pabs{\zeta^i-q^i}\leq \pabs{\zeta-q}=s< r$.
    The condition $i\leq\frac{1}{\lambda}\log_p\sog{\frac{1}{r}}$ is equivalent to
    \[r\leq p^{-\lambda i}=\pabs{\zeta-1}^i.\]
    Write $i=ap^k$ with $p\nmid a$. Then
    \[\pabs{1-\zeta^i}=\pabs{1-\zeta^{p^k}}=\pabs{1-\zeta}^{p^k}\geq \pabs{1-\zeta}^i\geq r.\]
    Thus,
    \[\pabs{1-q^i}=\pabs{(\zeta^i-q^i)+(1-\zeta^i)}=\abs{1-\zeta^i}_p\geq r.\]
    
    (2). We use part $(1)$ with $r=\sqrt{s}>s$.
    Then,
    \[\abs{1-q^i}_p=\abs{1-\zeta^i}_p\geq \sqrt{s}>s,\]
     for all $1\leq i\leq \frac{1}{\lambda}\log_p\sog{\frac{1}{\sqrt{s}}}$.
    If in addition $i>1$, then by writing $\zeta-q^i=(\zeta-q)+q(1-q^{i-1})$ we see that 
    \[\abs{\zeta-q^i}_p=\abs{q(1-q^{i-1})}_p=\abs{1-q^{i-1}}_p.\]
    Let $1<k\leq \frac{1}{\lambda}\log_p\sog{\frac{1}{\sqrt{s}}}$.
    Then        \begin{align*}
        \frac{\pcoeff{\zeta}{q}_k}{(q;q)_k}
    &=\frac{(\zeta-1)(\zeta-q)(\zeta-q^2)...(\zeta-q^{k-1})}{(1-q)(1-q^2)(1-q^3)...(1-q^k)}\\  
    &=\frac{(\zeta-1)(\zeta-q)}{(1-q^{k-1})(1-q^k)}\cdot\sog{\frac{\zeta-q^2}{1-q}}\cdot\sog{\frac{\zeta-q^3}{1-q^2}}\cdot...\cdot\sog{\frac{\zeta-q^{k-1}}{1-q^{k-2}}}
        \end{align*} 
    (Note that $(q;q)_k\neq 0$).
    Using the equality $\abs{\zeta-q^i}_p=\abs{1-q^{i-1}}_p$ for any $2\leq i\leq k-1$ we see that
    \[\frac{\abs{\pcoeff{\zeta}{q}_k}_p}{\abs{(q;q)_k}_p}=\frac{\abs{(\zeta-1)}_p\abs{(\zeta-q)}_p}{\abs{(1-q^{k-1})}_p\abs{(1-q^k)}_p}.\]
    By part $(1)$, $\abs{1-q^k}_p\geq \sqrt{s}$ and $\abs{1-q^{k-1}}_p\geq \sqrt{s}$.
    Moreover, since one of $k$ or $k-1$ is not divisible by $p$, the $p$-adic absolute value of one of them is equal to $\abs{1-q}_p$.
    The assumption that $\pabs{\zeta-1}>s$ implies that $\pabs{\zeta-1}=\pabs{q-1}$.  
    Thus, 
    \[\frac{\abs{\pcoeff{\zeta}{q}_k}_p}{\abs{(q;q)_k}_p}\leq \frac{\abs{\zeta-1}_p\cdot s}{\abs{\zeta-1}_p\cdot \sqrt{s}}=\sqrt{s}.\]
    \end{proof} 

    \begin{proof}[Proof of \ref{main_inequality}]
    Denote $\alpha=\frac{1}{2\lambda}\log_p\sog{\frac{1}{\sqrt{s}}}$ and let $N$ be large enough such that the following conditions are satisfied ($\zeta$ is a primitive $p^N$-th root of unity).
        \begin{enumerate}
        \item $\pabs{1-\zeta}>s$.
        \item $\alpha>p^8$.
        \item $\frac{\lambda}{4}\alpha\log_p(\alpha)\geq \log_p\sog{\frac{M}{\sqrt{s}}}$.
        Note that $\frac{\lambda}{4}\alpha\log_p(\alpha)=A\cdot \log_p(\frac{1}{\lambda})+B$ where $A=\frac{1}{8}\log_p\sog{\frac{1}{\sqrt{s}}}>0$ and $B=A\cdot\log_p\sog{\frac{1}{2}\log_p\sog{\frac{1}{\sqrt{s}}}}$.
        Note that $B$ does not depend on $\zeta$.
        \end{enumerate}
        
    Under these assumptions we will show that
    \[\norm{\pcoeff{\zeta}{q}_1\cdot\qbinom{x}{1}{q}}
    >\norm{\pcoeff{\zeta}{q}_k\cdot\qbinom{x}{k}{q}}\]
    for any $k\neq 1$.
    
    We begin by showing that $\norm{\qbinom{x}{1}{q}}=\norm{q^x}>1$.
    Indeed, by the assumption that $\pabs{\zeta-1}>s$ we have that $\pabs{\zeta-1}=\pabs{q-1}$.
    Thus,
    \[\norm{\pcoeff{\zeta}{q}_1\cdot\qbinom{x}{1}{q}}=\norm{(\zeta-1)\frac{1-q^x}{1-q}}=\norm{1-q^x}.\]
    By Proposition \ref{prop_norm_of_q^x}, $\norm{q^x}=G_{\norm{}}(s)>1$.
    Therefore, 
    \[\norm{1-q^x}=\norm{q^x}>1.\]
    
    Assume that $k=0$.
    Then \[\norm{\pcoeff{\zeta}{q}_0\cdot\qbinom{x}{0}{q}}=\norm{\textbf{1}_{\Z_p}(x)}=1.\]
    
    Assume $1<k\leq\frac{1}{\lambda}\log_p\sog{\frac{1}{\sqrt{s}}}$.
By part $(2)$ of Proposition \ref{some_p_adic_valuations} and by Proposition \ref{prop_norm_of_q^x},
        \begin{align*}
        \nrm*{\pcoeff{\zeta}{q}_k\cdot\qbinom{x}{k}{q}}
        =\frac{\abs{\pcoeff{\zeta}{q}_k}_p}{\abs{(q;q)_k}_p}\cdot\norm{(q^x-1)(q^x-q)...(q^x-q^{k-1})}
        \leq \sqrt{s}\cdot\norm{q^x}
        <\norm{q^x}
        =\norm{\pcoeff{\zeta}{q}_1\cdot\qbinom{x}{1}{q}}.
        \end{align*} 
        
    Assume that $k> \frac{1}{\lambda}\log_p\sog{\frac{1}{\sqrt{s}}}$.
    Let $m$ be an integer with 
    \[\frac{1}{2\lambda}\log_p\sog{\frac{1}{\sqrt{s}}}\leq m<\frac{1}{\lambda}\log_p\sog{\frac{1}{\sqrt{s}}}.\]
    Such an integer exists, since $\frac{1}{\lambda}\log_p\sog{\frac{1}{\sqrt{s}}}>2p^8> 4$.
    As $k>m$, $\abs{\pcoeff{\zeta}{q}_k}_p\leq\abs{\pcoeff{\zeta}{q}_m}_p$.
    By the second and first parts of Proposition \ref{some_p_adic_valuations} we have
    \[\abs{\pcoeff{\zeta}{q}_m}_p\leq \sqrt{s}\cdot\abs{(q;q)_m}_p=\sqrt{s}\cdot\abs{(\zeta;\zeta)_m}_p.\]
    Since $m>p^8$ we can apply Proposition \ref{q_analog_evaluations_cor}, and together with the assumption that $m\geq\alpha$ we get
        \begin{align*}
        \log_p(\abs{(\zeta;\zeta)_m}_p)
        \leq -\frac{\lambda}{4}m\log_p(m)
        \leq -\frac{\lambda}{4}\alpha\log_p(\alpha)
        \leq -\log_p\sog{\frac{M}{\sqrt{s}}}.
        \end{align*} 
    In the last inequality we used our assumption $3$.
    Then $\abs{(\zeta;\zeta)_m}_p\leq \frac{\sqrt{s}}{M}$.
    Therefore,
    \[\abs{\pcoeff{\zeta}{q}_k}_p\leq \sqrt{s}\cdot \abs{(\zeta;\zeta)_m}_p\leq \sqrt{s}\cdot \frac{\sqrt{s}}{M}=\frac{s}{M}.\]
    Finally,
    \[\norm{\pcoeff{\zeta}{q}_k\cdot\qbinom{x}{k}{q}}\leq \abs{\pcoeff{\zeta}{q}_k}_p\cdot M\leq s<1<\norm{\pcoeff{\zeta}{q}_1\cdot\qbinom{x}{1}{q}}.\]
    This completes the proof of \ref{main_inequality}, hence of Theorem \ref{thm_strong_minimality_Z_p} for $\Z_p$.
    \end{proof}

\section{Further discussion about $\mathcal{H}$-invariant norms on $\mathcal{S}$}
This section is motivated by the search for other minimal invariant norms on $\schw$.
In addition, Theorem \ref{prop_independence_minimal_norms} is a generalization of the discontinuity of the Fourier transform proved in \cite{ophir2016q} to finite families of intertwining operators.

By now we have constructed two types of $\heis$-invariant norms on $\schw$: the family of minimal norms $\braces{\norm{}_\alpha\ |\ \alpha\in \Grassmannian}$, and, for each non-zero $f\in\schw$, the maximal invariant norm at $f$ which we denoted by $\norm{}_f$.
The latter belong to the maximal equivalence class of $\heis$-invariant norms.
Given any subset $I\subset \Grassmannian$, we can form the norm $\sup_{\alpha \in I}\norm{}_\alpha$.
The supremum exists since all the $\norm{}_\alpha$, being normalized at $\textbf{1}_{\Z_p^d}(x)$, are bounded from above by the maximal invariant norm at $\textbf{1}_{p^d\Z_p}(x)$.
In this section we consider finite families $I\subset \Grassmannian$ and the norms
\[\norm{}_I:=\max_{\alpha\in I}\norm{}_\alpha,\]
and answer the question: are there new minimal norms that lie below $\norm{}_I$?
We show that the answer is negative.
In fact, we will show the following.
\begin{thm}\label{thm_classification_small_norms}
$\ $
    \begin{enumerate}
    \item Let $I,J\subset \Grassmannian$ be distinct finite subsets. Then $\norm{}_I$ and $\norm{}_J$ are not equivalent.
    \item Let $\norm{}$ be an $\heis$-invariant norm which is dominated by $\norm{}_I$, where $I\subset \Grassmannian$ is finite.
    Then there exists $J\subset I$ such that $\norm{}$ is equivalent to $\norm{}_J$.
    \item If $I_1,I_2\subset \Grassmannian$ are finite and disjoint, there does not exist any $\heis$-invariant norm on $\schw$ which is dominated by both $\norm{}_{I_1}$ and by $\norm{}_{I_2}$.
    \end{enumerate}
\end{thm} 
Clearly, $(1)$ implies that the $J\subset I$ in $(2)$ is unique, and $(1)$ and $(2)$ imply $(3)$.
If $L_\alpha$ is the unit ball of $\norm{}_{\alpha}$, the meaning of $(3)$ is that if we put $L_I=\bigcap_{\alpha\in I}L_\alpha$, then $L_{I_1}+L_{I_2}=\schw$.

We will also show that any norm of the form $\norm{}_I$, where $I\subset\Grassmannian$ is finite, is equivalent to a norm which is locally maximal at some vector.

To prove these results, we introduce a  notion of independence of norms.
\subsection{Independence of norms}
The setting in this sub-section is general.
Let $V$ be a vector space over $\C_p$.
\begin{prop}\label{Prop_independence_two_norms}
Let $\norm{}_1,\norm{}_2$ be two norms on $V$.
The following are equivalent.
    \begin{enumerate}
    \item There exists no (non-zero) seminorm on $V$ which is dominated by both $\norm{}_1$ and $\norm{}_2$.
    \item The diagonal map
    \[V\map \completion{V}{\norm{}_1}\oplus\completion{V}{\norm{}_2}\]
    has a dense image, where the norm on the right hand side is $(v,w)\mapsto \max(\norm{v}_1,\norm{w}_2)$.
    \item Let $L_1,L_2$ be the closed unit balls of $\norm{}_1,\norm{}_2$ respectively.
    Then $L_1+L_2=V$.
    \end{enumerate}
\end{prop} 
\begin{proof}
We will show that each of $(1)$ and $(2)$ is equivalent to $(3)$.
If 
To show that $(1)$ and $(3)$ are equivalent, note that the gauge of $L_1+L_2$ is either zero, if $L_1+L_2=V$, or defines a non-zero seminorm $\norm{}'$ on $V$.
The seminorm $\norm{}'$ is dominated by both $\norm{}_1$ and $\norm{}_2$, and any seminorm that is dominated by both $\norm{}_1$ and $\norm{}_2$ is also dominated by $\norm{}'$.
From this it follows that $(1)$ and $(3)$ are equivalent.

We now show that $(2)$ and $(3)$ are equivalent.
It is easy to see that $(2)$ is equivalent to the statement that for any $w\in V$ and $\eps>0$ there exists $v\in V$ such that $\norm{v-w}_1<\eps$ and $\norm{v}_2<\eps$.
This statement is equivalent to the claim that any $w\in V$ can be written as $w=v_1+v_2$ with $\norm{v_1}_1<\eps$ and $\norm{v_2}_2<\eps$, and this is equivalent to $(3)$.
\end{proof} 

\begin{defn}
We say that two norms $\norm{}_1,\norm{}_2$ on $V$ are \textit{independent} if one of the equivalent conditions of the previous proposition is satisfied.
\end{defn} 

\begin{defn}
We say that the norms $\norm{}_1,...,\norm{}_n$ on $V$ are independent if for any $1\leq i\leq n$ the two norms:
\[\norm{}_i \ \ \ \text{and}\ \ \   \max_{\substack{1\leq j\leq n \\ j\neq i}}\norm{}_j\]
are independent.
\end{defn} 

Note that if $\norm{}_1,...,\norm{}_n$ are independent, so is any subset of them.

\begin{prop}\label{prop_independence_general}
Let $\norm{}_1,...,\norm{}_n$ be norms on $V$.
The following are equivalent.
    \begin{enumerate}
    \item $\norm{}_1,...,\norm{}_n$ are independent.
    \item The diagonal embedding
    \[V\xmap{\ \triangle\ }\bigoplus_{i=1}^n\completion{V}{\norm{}_i}\]
    has a dense image.
    \item For any two disjoint sets $I,J\subset \braces{1,2,..,n}$ the norms 
    \[\max_{i\in I}\norm{}_i\ \ \ \text{and}\ \ \ \max_{j\in J}\norm{}_j\]
    are independent.
    \end{enumerate}
\end{prop} 
    \begin{proof}
    As $(1)$ is a particular case of $(3)$, it remains to show $(1)\Rightarrow(2)\Rightarrow(3)$.
    We will prove this by induction on $n$.
    The case $n=2$ is essentially Proposition \ref{Prop_independence_two_norms}.
    
    Assume $(1)$.
    By the assumption and Proposition \ref{Prop_independence_two_norms}, the diagonal map
    \[V\map \completion{V}{\norm{}_1}\oplus\completion{V}{\max_{1<i\leq n}\norm{}_i}\]
    has a dense image.
    The norms $\norm{}_2,...,\norm{}_n$ are also independent and by the induction hypothesis the map
    \[\completion{V}{\max_{1<i\leq n}\norm{}_i}\map \bigoplus_{1<i\leq n}\completion{V}{\norm{}_i}\]
    is an isomorphism.
    Thus, $V\map \bigoplus_{1\leq i\leq n}\completion{V}{\norm{}_i}$ also has a dense image.
    
    Assume $(2)$ and let $I,J\subset \braces{1,...,n}$ be non-empty disjoint subsets.
    We may assume that $I\cup J=\braces{1,...,n}$.
    Denote $\norm{}_I=\displaystyle\max_{i\in I}\norm{}_i$ and $\norm{}_J=\displaystyle\max_{j\in J}\norm{}_j$.
    Consider the maps:
    \[V\xmap{\triangle} \completion{V}{\norm{}_I}\oplus\completion{V}{\norm{}_J}\map \bigoplus_{1\leq i\leq n}\completion{V}{\norm{}_i}.\]
    By the induction hypothesis, the second arrow is an isometric isomorphism. The first arrow $\Delta$ therefore has a dense image, so $(3)$ follows from Proposition \ref{Prop_independence_two_norms}.
    \end{proof} 

The following Proposition is left as an exercise to the reader.
\begin{prop}\label{prop_ind_invariant_norms}
Assume that $V$ is an irreducible representation of a group $G$, and $\norm{}_1,\norm{}_2\in\Norms(V)^G$.
Then $\norm{}_1$ and $\norm{}_2$ are independent if and only if there exists no $G$-invariant norm on $V$ that is dominated by both $\norm{}_1$ and $\norm{}_2$.
\end{prop}

\subsection{Proofs of the claims in this section}

\begin{prop}\label{prop_independence_minimal_norms}
Let $I\subset \Grassmannian$ be a finite subset.
The norms $\braces{\norm{}_\alpha\ |\ \alpha\in I}$ are independent.
\end{prop} 
    \begin{proof}
    The proof is by induction on the size of the set $I$.
    If $|I|=1$ there is nothing to prove.
    Assume that $|I|=n>1$.
    Let $\alpha\in I$, we need to show that the two norms
    \[\norm{}_\alpha\ \ \ \text{and}\ \ \ \norm{}_{I\backslash\braces{\alpha}}:=\max_{\beta\in I\backslash\braces{\alpha}}\norm{}_\beta\]
    are independent.
    By Theorem \ref{thm_strong_minimality} and Proposition \ref{prop_ind_invariant_norms} it is enough to prove that $\norm{}_{I\backslash\braces{\alpha}}$ does not dominate $\norm{}_\alpha$.
    Suppose, for a contradiction, that $\norm{}_\alpha \dominated \norm{}_{I\backslash\braces{\alpha}}$.
    By the induction hypothesis, there is an isometry
    \[\completion{\schw}{I\backslash\braces{\alpha}}\xmap{\sim} \bigoplus_{\beta\in I\backslash\braces{\alpha}}\completion{\schw}{\norm{}_\beta}.\]
    Thus, we obtain a non-zero map 
    \[\bigoplus_{\beta\in I\backslash\braces{\alpha}}\completion{\schw}{\norm{}_\beta}\map \completion{\schw}{\norm{}_\alpha}.\]
    Then there exists $\beta\in I\backslash\braces{\alpha}$ such that the reduced map $\completion{\schw}{\norm{}_\beta}\map \completion{\schw}{\norm{}_\alpha}$ is non-zero.
    By Proposition \ref{topologically_irreducible} we have $\alpha=\beta$, a contradiction.    
    \end{proof} 

\begin{cor}
Let $I\subset \Grassmannian$ be a finite subset.
The norm $\norm{}_I$ is equivalent to a locally maximal norm (with respect to some vector).
\end{cor} 
    \begin{proof}
    Since the norms $\braces{\norm{}_\alpha\ |\ \alpha\in I}$ are independent, it follows by Proposition \ref{prop_independence_general} that
    \[\completion{\schw}{\norm{}}\simeq \bigoplus_{\alpha\in I}\completion{\schw}{\norm{}_\alpha}\]
    are isomorphic Banach representations (and even isometrically isomorphic).
    By Proposition \ref{topologically_irreducible}, the spaces $\braces{\completion{\schw}{\norm{}_\alpha}\ |\ \alpha\in I}$ are pairwise non-isomorphic.
    By Proposition \ref{prop_small_representations}, $\bigoplus_{\alpha\in I}\completion{\schw}{\norm{}_\alpha}$ has a strongly cyclic vector, and by Theorem \ref{thm_strongly_cyclic_spaces}, $\norm{}$ is equivalent to a locally maximal norm.
    \end{proof} 

\begin{proof}[Proof of Theorem \ref{thm_classification_small_norms}]
    As already noted, $(3)$ follows from $(1)$ and $(2)$.
    $(1)$ follows from the fact that the $\norm{}_\alpha$ are independent (Proposition \ref{prop_independence_minimal_norms}) and by Proposition \ref{prop_independence_general}.
    We prove $(2)$ by induction on the size of $I$.
    When $|I|=1$ the claim follows from Theorem \ref{thm_strong_minimality}.
    Assume that $|I|=n>1$, and that the claim is true for all subsets of $\Grassmannian$ of size $<n$.
    Let $\norm{}$ be an $\heis$-invariant norm on $\schw$ that is dominated by $\norm{}_I$.
    Then $\norm{}$ extends to an $\heis$-invariant seminorm on the completion $\completion{\schw}{\norm{}_I}$, which, by the independence of the $\norm{}_\alpha$, is isometrically isomorphic to $\bigoplus_{\alpha\in I}\completion{\schw}{\norm{}_\alpha}$ via the diagonal embedding.
    By Proposition \ref{prop_small_representations}, the kernel of $\norm{}$ is of the form $\bigoplus_{\alpha\in K}\completion{\schw}{\norm{}_{\alpha}}$ for some subset $K\subset I$.
    Using the diagonal embedding, this means that $\norm{}$ is already dominated by $\norm{}_{I\backslash K}$.
    If $K$ is non-empty, then $|I\backslash K|<|I|$ and the claim is true by the induction hypothesis.
    Assume that $K$ is empty.
    Then $\norm{}$ is a norm on $\bigoplus_{\alpha\in I}\completion{\schw}{\norm{}_{\alpha}}$.
    We want to show that in this case $\norm{}$ is equivalent to $\norm{}_I$.
    Choose $\alpha\in I$ and denote $J=I\backslash\braces{\alpha}$.
    By Theorem \ref{thm_strong_minimality}, the restriction of $\norm{}$ to the component $\completion{\schw}{\norm{}_\alpha}$ is of the form $r_\alpha\cdot \norm{}_\alpha$.
    Similarly, the seminorm on $\completion{\schw}{\norm{}_\alpha}$, obtained from $\norm{}$ by taking the quotient of $\bigoplus_{\alpha\in I}\completion{\schw}{\norm{}_\alpha}$ by $\bigoplus_{\beta\in J}\completion{\schw}{\norm{}_\beta}$ is of the form $s_\alpha\cdot \norm{}_\alpha$.
    Clearly, $s_\alpha \leq r_\alpha$.
    We claim that $0<s_\alpha$.
    By the induction hypothesis, the restriction of $\norm{}$ to the component $\bigoplus_{\beta\in J}\completion{\schw}{\norm{}_\beta}$ is equivalent to $\norm{}_J$.
    It follows that $\bigoplus_{\beta\in J}\completion{\schw}{\norm{}_\beta}$ is a closed subspace of $\bigoplus_{\alpha\in I}\completion{\schw}{\norm{}_\alpha}$ with respect to the topology induced by $\norm{}$.
    Therefore, $s_\alpha\cdot \norm{}_\alpha$ is a norm, so $s_\alpha>0$.
    This is true for any $\alpha\in I$, so
    \[\max_{\alpha\in I}(s_\alpha\cdot\norm{}_\alpha)\leq\norm{}'\leq \max_{\alpha\in I}(r_\alpha\cdot\norm{}_\alpha),\] 
    which shows that $\norm{}$ is equivalent to $\norm{}_I$.
   \end{proof}

\subsection{Open questions}
We conclude with some open questions that we find interesting.
\begin{question}
Does there exists an $\heis$-invariant norm on $\schw$ which does not dominate any of the norms $\norm{}_\alpha$, for $\alpha\in\Grassmannian$?
\end{question}
We find this question especially interesting, regardless of the answer.
If the answer is negative, the spaces $\braces{\completion{\schw}{\norm{}_\alpha}\ |\ \alpha\in \Grassmannian}$ form a complete list of the irreducible completions of $\schw$.
If the answer is positive, constructing such norms will require new ideas that could be useful in the study of Banach representations of $p$-adic groups.
In the latter case, we also ask
\begin{question}
Does there exist another $\heis$-invariant norm on $\schw$, the completion by which is an (strongly) irreducible Banach representations?
\end{question}

The last section gives a complete picture of those norms which are dominated by some $\norm{}_I$, for a finite subset $I\subset \Grassmannian$.
When $I$ is not finite, we can still define the norm $\norm{}_I$ as before.
Now it seems reasonable to consider the topology of $\Grassmannian$.

\begin{question}
Let $I_1,I_2$ be closed and disjoint subsets of $\Grassmannian$.
    \begin{enumerate}
    \item Are the norms $\norm{}_{I_1}$ and $\norm{}_{I_2}$ independent?
    \item Is there a simple description of the completion $\completion{\schw}{\norm{}_{I_1}}$ in terms of the completions $\completion{\schw}{\norm{}_\alpha}$ for $\alpha\in I_1$?
    \end{enumerate}
\end{question}
Finally, taking $I=\Grassmannian$, we ask
\begin{question}
Does the norm $\sup_{\alpha\in\Grassmannian}\norm{}_\alpha$ belong to the maximal equivalence class of $\heis$-invariant norms on $\schw$?
\end{question}

\bibliographystyle{abbrv}

\end{document}